\documentclass[12pt]{article}
\usepackage[letterpaper, margin=1in]{geometry}

\usepackage{amsmath, amssymb, amsthm, color, mathtools}
\usepackage{caption, graphicx}
\usepackage{subcaption}
\usepackage{enumitem}
\newtheorem{theorem}{Theorem}
\newtheorem{corollary}{Corollary}[theorem]

\usepackage[longnamesfirst, authoryear]{natbib}
\usepackage{authblk}
\usepackage{tabularray}
\newtheorem{definition}{Definition}
\newtheorem{lemma}[theorem]{Lemma}

\newtheorem{remark}{Remark}

\newcommand{\tG}{\theta_G}

\usepackage{changepage}
\usepackage[utf8]{inputenc}
\usepackage{hyperref}
\usepackage{setspace}
\hypersetup{hidelinks,colorlinks=true,linkcolor=blue,citecolor=blue,linktocpage=true}

\providecommand{\keywords}[1]
{
  {\textit{Keywords---} #1}
}

\usepackage[title, titletoc]{appendix}
\title{On the Precise Asymptotics of Universal Inference}
\author{Kenta Takatsu}
\affil{Department of Statistics and Data Science, Carnegie Mellon University}
\date{}

\DeclareMathOperator{\E}{\mathbb{E}}

\DeclareMathOperator*{\argmin}{arg\,min}
\usepackage[longnamesfirst, authoryear]{natbib}
\usepackage{algorithm}
\usepackage{algpseudocode}
\usepackage{tocloft}
\cftsetindents{section}{0em}{3em}
\cftsetindents{subsection}{3em}{3em}
\begin{document}
\maketitle

\begin{abstract}
In statistical inference, confidence set procedures are typically evaluated based on their validity and width properties. Even when procedures achieve rate-optimal widths, confidence sets can still be excessively wide in practice due to elusive constants, leading to extreme conservativeness, where the empirical coverage probability of nominal $1-\alpha$ level confidence sets approaches one. This manuscript studies this gap between validity and conservativeness, using \emph{universal inference} \citep{wasserman2020universal} with a regular parametric model under model misspecification as a running example. We identify the source of asymptotic conservativeness and propose a general remedy based on \emph{studentization} and \emph{bias correction}. The resulting method attains exact asymptotic coverage at the nominal $1-\alpha$ level, even under model misspecification, provided that the product of the estimation errors of two unknowns is negligible, exhibiting an intriguing resemblance to double robustness in semiparametric theory.
\end{abstract}
\keywords{Universal Inference, Central Limit Theorem, Berry-Esseen Bound, Model Misspecification, Studentization, Double Robustness}
\tableofcontents
\section{Introduction}
Traditional statistical inference techniques, such as likelihood ratio tests, have seen renewed interest in recent years, driven in part by the growing emphasis on methodologies based on \emph{e-values} and \emph{e-processes}, rather than conventional \emph{p-values}. Unlike p-values, e-values possess several properties that make them particularly appealing for modern data science applications. In particular, e-value-based methods have played an instrumental role in advancing multiple and safe testing \citep{grunwald2020safe, Vovk2021evalues, Shafer2021Testing, wang2022false}, anytime-valid inference \citep{waudby2024estimating}, and asymptotic confidence sequences \citep{waudby2024time}. This list is far from exhaustive, and we refer to \citet{Ramdas2023Test} for a broader overview of recent developments.

This manuscript revisits the work of \citet{wasserman2020universal}, who introduced \emph{universal inference}, a general hypothesis testing framework based on \emph{split likelihood ratio statistics}, which is also an e-value. This framework provides simple procedures for many complex composite testing problems that previously lacked actionable solutions, such as testing log-concavity \citep{dunn2024universal} and causal inference under unknown causal structures \citep{strieder2021confidence}, among others. Specifically, universal inference combines the classical idea of sample splitting \citep{Cox1975Note} and Markov's inequality to establish finite-sample validity. The procedure follows three steps. Suppose independent and identically distributed (IID) observations $Z_1, \ldots, Z_n$ follow the distribution $P_{\theta_0}$, belonging to a collection of distributions $\{P_{\theta} \, :\, \theta \in \Theta\}$, for an arbitrary set $\Theta$. The procedure is as follows: (1) Split the index set $\{1, 2, \ldots, n\}$ into two non-overlapping sets $D_1$ and $D_2$; (2) Construct an arbitrary initial estimator $\widehat\theta_1$ using the first data $\{Z_i\, : \, i \in D_1\}$ and (3) Define the confidence set as
 \begin{align}\label{eq:ui-def}
     \mathrm{CI}^{\mathrm{UI}}_{n, \alpha} = \left\{\theta \in \Theta \, : \, \sum_{i \in D_2}\,\log \frac{p_{\widehat\theta_1}(Z_i)}{p_{\theta}(Z_i)} \le \log(1/\alpha) \right\}.
 \end{align}
The confidence set constructed based on universal inference is henceforth referred to as a \emph{universal confidence set}\footnote{We retain this term for simplicity, despite its unfortunate naming collision with another universal confidence set proposed by \citet{vogel2008universal}.}. A striking feature of this method is that it imposes almost no regularity conditions on the data-generating distribution, which are typically required to ensure validity in large-sample theory. Perhaps unsurprisingly, such robustness often comes at the cost of conservativeness. Shortly after the publication of \citet{wasserman2020universal}, several studies \citep{tse2022note, strieder2022choice, dunn2023gaussian} highlighted the conservative nature of this approach.

Interestingly, universal inference can be viewed as unimprovable in a certain sense. First, its validity relies on Markov’s inequality, which is tight for some distributions with two-point supports. Second, the diameter of the universal confidence set, under additional regularity conditions, can be shown to converge at an optimal rate (see Theorem 6 of \citet{wasserman2020universal} and Remark~\ref{remark:power-ui}). Despite these seemingly favorable properties, \citet{tse2022note} and many others report that the empirical coverage probability of a $1-\alpha$ level universal confidence set can be close to one, rendering the method overly conservative in practice. 

In addition to conservativeness, another fundamental issue arises when the likelihood is misspecified. The universal confidence set \eqref{eq:ui-def} is constructed under the assumption that the data-generating distribution belongs to the statistical model at hand. When this assumption is violated, one might hope that the universal confidence set remains valid for the best approximation within the working model. Unfortunately, that is not the case. \citet{park2023robust} and \citet{dey2024anytime} examined this aspect, showing that regaining validity under misspecification requires either strong additional regularity conditions or methodological adjustments.

One of the goals of this manuscript is to formally examine this gap between validity and extreme conservativeness in universal inference under (misspecified) regular models. At this point, a reader may wonder why we analyze universal inference under regularity conditions when its appeal lies in avoiding such assumptions for validity. In line with \citet{strieder2022choice, dunn2023gaussian}, we impose these conditions to develop a more precise understanding of its properties. In doing so, we suggest a broader story regarding the trade-off between regularity and conservativeness, even when the resulting method remains valid and rate-optimal under weak assumptions. The analysis in the manuscript suggests that, without modifications leveraging additional problem structures, conservativeness can be often extreme. Since such modifications have received little attention, extreme conservativeness has been a pervasive limitation among universal inference and its derivatives.

The two aforementioned issues of over-conservativeness and model misspecification can be addressed by slight changes in perspective. The conservativeness of universal inference might be mitigated if one is willing to give up finite-sample validity in favor of asymptotic guarantees, while model misspecification can be tackled by framing the problem as an inference for the solution to an optimization problem. Indeed, sample splitting has been explored for asymptotic inference for high-dimensional or irregular problems \citep{park2023robust, kim2020dimension, takatsu2025bridging}. Similar to universal inference, these methods rely on sample splitting but differ crucially in their theoretical justification; Instead of invoking Markov's inequality, they use the central limit theorem (CLT) to establish asymptotic validity. The required regularity conditions are relatively mild, roughly aligning with those necessary for the CLT for Student's t-statistics \citep[Corollary 1.5]{Bentkus1996}. By leveraging the CLT, it may be conceivable that these methods are more powerful and (asymptotically) less conservative than universal inference or other e-value-based methods, as they can more precisely estimate the quantile of the test statistics. 

On the other hand, inference for the solution to optimization problems has a long history that well precedes universal inference. In the context of maximum likelihood estimation (MLE) under misspecification, asymptotic inference can be performed using the sandwich variance estimator \citep{Huber1967}. Although often overlooked in the statistical inference literature, inference for more general optimization problems has been extensively studied in operations research, with key references including \citet{vogel2008confidence, vogel2008universal, vogel2017confidence, vogel2019universal}. More recently, \citet{takatsu2025bridging} extended the sample splitting approach to the general framework presented by \citet{vogel2008universal}, further relaxing regularity conditions and demonstrating validity for a broader class of distributions. Furthermore, the width of their confidence set exhibits an adaptive property, automatically converging at a faster rate depending on the local regularity of the optimization problem near the solution. 


Some contributions of the manuscript are as follows: (1) We provide new theoretical results on universal inference under model misspecification, making a formal statement about conditions under which the method becomes conservative. The analysis is performed under standard regularity conditions. While these conditions might be sufficient for the asymptotic normality of the MLE under model misspecification for fixed parameter dimensions, they may not suffice for Wald inference when the dimension is large; (2) For the variant of universal inference known as generalized universal inference \citep{dey2024anytime}, we present results showing that coverage validity does not require the \emph{strong central condition}. The possibility of weaker assumptions was already alluded to in \citet{dey2024anytime} and \citet{takatsu2025bridging}; (3) We present results indicating that asymptotic conservativeness is far more pervasive across universal inference and its derivatives, such as those proposed by \citet{nguyen2020universal}, as well as e-value-based procedures. Without any corrections, the empirical coverage probability at any fixed confidence level $1-\alpha$ can asymptotically converge to one, a phenomenon we term \emph{extreme conservativeness}. While the conservative property was mentioned in \citet{wasserman2020universal} and \citet{park2023robust}, we present a formal explanation of this phenomenon and complete the story; 
(4) To mitigate conservative coverage, we develop a method based on \emph{studentization} and \emph{bias correction}. An interesting second-order bias condition emerges, akin to those found in the double robustness literature \citep{bickel1982adaptive, pfanzagl1985contributions, klaassen1987consistent, robinson1988root, bickel1993efficient, laan2003unified}. The proposals found in this manuscript are fundamentally different from previously considered power-enhancement methods, such as cross-fitting \citep{newey2018cross}, randomization \citep{ramdas2023randomized}, or majority voting \citep{gasparin2024merging}, which do not address the problem at its core.

\section{Definitions and Main Claims}
Let $\mathcal{P}$ denote a set of probability distributions over some measurable space. Given observations $Z_1, \ldots, Z_n$, each
distributed according to a probability measure $P \in \mathcal{P}$, the objective is to perform inference on the evaluation of a functional $\theta : \mathcal{P} \mapsto \mathbb{B}$ at $P$, where $\mathbb{B}$ is most often a subset of $\mathbb{R}^d$. A confidence set procedure is a sequence of (random) sets in $\mathbb{B}$, denoted by $\{\mathrm{CI}_{n, \alpha} \, :\, n \ge 1\}$, computed from $n$ observations. Below, we define key properties of confidence set procedures that are often considered in the literature.
\begin{definition}\label{def:uniformly-valid}
A sequence of confidence sets $\{\mathrm{CI}_{n, \alpha} \, :\, n \ge 1\}$ for a functional $\theta(P)$ is (asymptotically) uniformly valid at fixed level $\alpha \in [0,1]$ for a class of distributions $\mathcal{P}$ if
    \begin{align}\label{eq:uniformly-valid}
        \limsup_{n\to \infty}\, \sup_{P \in \mathcal{P}}\, \bigg(\mathbb{P}_P(\theta(P) \not\in \mathrm{CI}_{n, \alpha}) - \alpha\bigg)_+ = 0.
    \end{align}
\end{definition}

\begin{definition}\label{def:uniformly-exact}
A sequence of confidence sets $\{\mathrm{CI}_{n, \alpha} \, :\, n \ge 1\}$ for a functional $\theta(P)$ is (asymptotically) uniformly exact at fixed level $\alpha \in [0,1]$ for a class of distributions $\mathcal{P}$ if
    \begin{align}\label{eq:uniformly-exact}
        \limsup_{n\to \infty}\, \sup_{P \in \mathcal{P}}\, \big|\, \mathbb{P}_P(\theta(P) \not\in \mathrm{CI}_{n, \alpha}) - \alpha \, \big| = 0.
    \end{align}
\end{definition}
The difference between these definitions is crucial. Uniform validity (Definition~\ref{def:uniformly-valid}) only considers the positive part of the miscoverage discrepancy, and does not systematically regulate over-conservativeness. On the other hand, uniform exactness (Definition~\ref{def:uniformly-exact}) ensures the confidence set asymptotically attains the exact nominal coverage level, preventing both under-coverage and over-coverage.

Definition~\ref{def:uniformly-valid}, commonly referred to as honest inference \citep{Li1989Honest, Ptscher2002Lower}, has been extensively studied. \citet[Definition 5]{Kuchibhotla2023Median} defines a stricter, and hence more robust, validity condition where the supremum is also taken over all $\alpha \in [\alpha_n, 1]$ for an appropriately chosen sequence $\alpha_n \to 0$. In contrast, we consider a weaker notion of validity where $\alpha$ is assumed to be fixed and independent of the data. 

Perhaps less studied is the following notion of \emph{extreme conservativeness}:
\begin{definition}\label{def:uniformly-conservative}
A sequence of confidence sets $\{\mathrm{CI}_{n, \alpha} \, :\, n \ge 1\}$ for a functional $\theta(P)$ is (asymptotically) uniformly extremely conservative at fixed level $\alpha \in [0,1]$ for a class of distributions $\mathcal{P}$ if
    \begin{align}\label{eq:uniformly-conservative}
        \limsup_{n\to \infty}\, \sup_{P \in \mathcal{P}}\, \mathbb{P}_P(\theta(P) \not\in \mathrm{CI}_{n, \alpha}) = 0.
    \end{align}
\end{definition}
As the remainder of this manuscript often focuses on the large-sample properties of confidence set procedures, we occasionally omit the term ``asymptotically" from the definitions. Extreme conservativeness is generally undesirable, particularly when $\mathcal{P}$ is considerably large. However, since extremely conservative confidence sets are valid by definition, this issue often goes unnoticed, as validity can sometimes be trivially established using tools such as Markov’s inequality. In fact, we show that extreme conservativeness is largely unavoidable unless additional assumptions or modifications are imposed.
\paragraph{Main claims}
We demonstrate that the nonasymptotic miscoverage probability of the confidence set under investigation can be expressed as follows:
\begin{align}
        |\,\mathbb{P}_P(\theta(P) \not\in \mathrm{CI}_{n, \alpha}) - \big(1-\Phi(z_\alpha + r_{n,P})\big) \,| \le  \Delta_{n, P},
\end{align}
where $\Phi(\cdot)$ is the cumulative distribution function of the standard normal random variable, and $z_\alpha$ is the $1-\alpha$ quantile of the standard normal distribution. There are two remainder terms: $r_{n,P}$ and $\Delta_{n,P}$. The second term $\Delta_{n, P}$ measures the accuracy of the Gaussian approximation, which becomes negligible when $\Delta_{n, P} \to 0$. Provided that $\Delta_{n, P} \to 0$, the term $r_{n,P}$ quantifies (anti-)conservativeness, where $r_{n,P} \to 0$ corresponds to exact coverage, while $r_{n,P} \to \infty$ corresponds to extreme conservativeness, i.e., asymptotic miscoverage probability approaches zero. Whenever $r_{n,P}$ is negative, the method becomes asymptotically invalid. With the proposed modifications based on studentization and bias correction, we show that $r_{n,P} \to 0$ is achieved as long as the product of the estimation errors of two unknowns (to be defined precisely) is negligible. This condition inherently ties to the dimension of  $\theta(P)$, suggesting that avoiding extreme conservativeness is likely a dimension-dependent property, without additional assumptions to the problem structures. 

\paragraph{Scope of this manuscript}
The remainder of this manuscript uses the confidence set for the population MLE under model misspecification as a running example. We begin by revisiting universal inference and identifying the source of conservativeness. We then propose a new procedure, building on the framework of \citet{takatsu2025bridging}, which is asymptotically valid, non-conservative, and exhibits optimal convergence rates. While recent studies \citep{park2023robust, dey2024anytime} have examined universal inference under model misspecification, this manuscript focuses on the precise quantification of conservativeness across a broader class of distributions as well as its remedy, giving rise to a new method. We hope this contribution will prove to be non-trivial.

We propose a general remedy to universal inference through studentization and bias correction. Although our primary focus is on applying this methodological modification to the MLE under model misspecification, we believe this approach can be broadly applied to other inferential methods that rely on e-values. While we do not fully pursue this direction, we hope that this general construction will contribute to addressing the empirically observed conservativeness in e-value-based procedures.

\paragraph{Notation}
We adopt the following convention. For $x \in \mathbb{R}^d$, we write $\|x\| = \sqrt{x^\top x}$. In particular, we define the unit sphere with respect to $\|\cdot\|$ such that $\mathbb{S}^{d-1} = \{u \in \mathbb{R}^d \, : \, \|u\| =1\}$. Given a square matrix $A \in \mathbb{R}^{d\times d}$, the smallest and the largest eigenvalues are denoted by $\lambda_{\min}(A)$ and $\lambda_{\max}(A)$ respectively. Additionally, for a positive semi-definite matrix $A \in \mathbb{R}^{d\times d}$ we denote the scaled norm as $\|x\|_A^2 = x^\top Ax$. We also denote by $\|\cdot\|_{\mathrm{op}}$ the operator norm of the matrices. 
For a generic distribution $P$, we denote probability, expectation, and variance under $P$ by $\mathbb{P}_P(\cdot)$, $\E_P[\cdot]$, and $\mathrm{Var}_P[\cdot]$, respectively.

\section{Maximimum Likelihood under Misspecification}
We begin by defining the setup and introducing key regularity conditions. For $\Theta \subseteq \mathbb{R}^d$, let us define a statistical model 
 $\mathcal{P}_{\Theta} := \{P_\theta \, : \, \theta\in \Theta \}$, dominated by a $\sigma$-finite base measure $\mu$. Throughout, we denote the density with respect to $\mu$ by $p_\theta := dP_\theta / d\mu$. Suppose we have IID observations $Z_1, \ldots, Z_n$ from a data-generating distribution $G$, which may not necessarily belong to the model $\mathcal{P}_{\Theta}$. The target of inference is the parameter $\theta_G$, indexing the distribution $P_{\theta_G} \in \mathcal{P}_{\Theta}$ which minimizes the Kullback–Leibler (KL) divergence between $G$ and the model $\mathcal{P}_{\Theta}$. This parameter is defined by the following KL-projection functional such that 
\begin{align}\label{eq:optimization}
    \theta_G \equiv \theta(G):= \argmin_{\theta \in \Theta}\, \mathrm{KL}(G \| P_{\theta}).
\end{align}
When $G \in \mathcal{P}_{\Theta}$, then it follows that $G = P_{\theta_G}$ $\mu$-almost everywhere. The notation introduced above is summarized in Table~\ref{tab:notation} in the supplement for the reader's convenience. 

Below, we introduce the concept of quadratic mean differentiability (QMD) under model misspecification, which imposes smoothness on the working model $\mathcal{P}_{\Theta}$ while allowing the data-generating distribution $G$ to be non-smooth. We present a nonasymptotic formulation for the definition. Other nonasymptotic definitions can be found, for instance, in \citet{Spokoiny2012Parametric, Spokoiny2015Finite}.
\begin{definition}[QMD under Misspecification] Let $\{\delta_n\}_{n=1}^\infty$ be a non-negative sequence such that $\delta_n \to 0$ as $n \to \infty$. A parametric model $\mathcal{P}_{\Theta}$ is said to be QMD at $\theta_G$ if there exist a $G$-square integrable score function $\dot\ell_{G} : \mathcal{Z} \mapsto \mathbb{R}^d$ and a  modulus of continuity $\phi_G : \mathbb{R}_+ \mapsto \mathbb{R}_+$ such that for any $\delta \le \delta_n$, 
\begin{align}\label{eq:dqm}
    \sup_{\|h\|_{I_{\tG}} \le \delta}\,\E_G \left[\left(\frac{p^{1/2}_{\theta_G+h}}{p^{1/2}_{\tG}}-1-\frac{1}{2}h^\top \dot\ell_{G}\right)^2\right] \le \phi_G(\delta),
\end{align}
and $\phi_G(\delta)/\delta^2 \to 0$ as $n \to \infty$. The Fisher information matrix is defined as $I_{G} := \E_{G}[\dot\ell_{G}\dot\ell_{G}^\top]$.
\end{definition}
When $G=P_{\theta_G}$ holds $\mu$-almost everywhere, this condition reduces to the traditional QMD assumption found in the literature \citep{Lecam1970Assumptions}. Notably, QMD assumption does not require pointwise differentiability of the density function $p_{\theta}$. This type of QMD conditions further implies stochastic local asymptotic normality (LAN) of the log-likelihood ratio, which has also been studied in the context of model misspecification \citep{Kleijn2006Infinite, Kleijn2012Bernstein, Bickel2012semiparametric}. Importantly, this stochastic LAN framework can be extended beyond the IID setting.

The regularity conditions laid out so far have primarily concerned the working statistical model $\mathcal{P}_{\Theta}$. Next, we introduce a regularity condition on the class of distributions $\mathcal{G}$, which contains the data-generating distribution $G$. Roughly speaking, the following condition ensures that the CLT can be invoked for the IID random variables $\{\langle t, \dot\ell_{G}(Z_i)\rangle\, :\, i \in D_2\}$ for all $t \in \mathbb{S}^{d-1}$ where $\dot\ell_{G}$ is the score function in \eqref{eq:dqm}. Below, for any real-valued random variable $X$, we write $\|X\|_p = (\mathbb{E}[|X|^p])^{1/p}$ for any $p > 0$. 
\begin{definition}[Uniform Lindeberg condition]
    A distribution $P$ supported on (a subset of) $\mathbb{R}^d, d\ge1$ is said to satisfy the {\em uniform Lindeberg condition} (ULC) if a mean-zero random variable $H\sim P$  satisfies  
    \begin{align}
        \lim_{\kappa \to \infty}\, \sup_{t \in \mathbb{R}^d}\,  \E\left[\frac{\langle t, H \rangle^2}{\|\langle t, H \rangle\|_{2}^2} \min\left\{1, \frac{|\langle t, H \rangle|}{\kappa \|\langle t, H \rangle\|_{2}}\right\}\right]  = 0.
    \end{align}
    In particular, a class of distributions $\mathcal{P}$ is said to satisfy the ULC if   
     \begin{align}\label{eq:validity-condition-clt}
        \lim_{\kappa \to \infty}\,\sup_{P\in\mathcal{P}}\, \sup_{t \in \mathbb{R}^d}\,  \E_P\left[\frac{\langle t, H \rangle^2}{\|\langle t, H \rangle\|_{2}^2} \min\left\{1, \frac{|\langle t,H\rangle|}{\kappa \|\langle t, H \rangle\|_{2}}\right\}\right]  = 0.
    \end{align}
\end{definition}
This condition ensures that the quantitative bound of the CLT, often referred to as the Berry-Esseen bound, becomes negligible as $n \to \infty$ \citep{katz1963note}. As noted in \citet{takatsu2025bridging}, the ULC is sufficient, but not necessary, for the convergence of the Berry-Esseen bound. Nevertheless, the ULC is more general than the conditions previously explored, for instance, by \citet[Lyapunov ratio]{kim2020dimension} or Theorem 3 of \citet{park2023robust}. Other weaker conditions were considered in \citet{park2023robust} by analyzing projection-functionals based on metrics other than the KL-divergence. The following subsections will present the main results.

\subsection{Extreme Conservativeness of Universal Inference}
To recall, the universal confidence set is defined as follows. Given IID observations $Z_1, \ldots, Z_N$ following the distribution $G$, (1) split the index set $\{1, 2, \ldots, N\}$ into two non-overlapping sets $D_1$ and $D_2$, with $|D_2| = n$; (2) construct an 
 arbitrary initial estimator $\widehat\theta_1$ using $\{Z_i\, : \, i \in D_1\}$; and (3) define the confidence set as
 \begin{align*}
     \mathrm{CI}^{\mathrm{UI}}_{n, \alpha} = \left\{\theta \in \Theta \, : \, \sum_{i \in D_2}\,\log \frac{p_{\widehat\theta_1}(Z_i)}{p_{\theta}(Z_i)} \le \log(1/\alpha) \right\}.
 \end{align*}
Theorem 1 of \citet{wasserman2020universal} proves that the confidence set $\mathrm{CI}^{\mathrm{UI}}_{n, \alpha}$ is finite-sample valid wherever $G \in \mathcal{P}_\Theta$. However, as discussed in Section 3 of \citet{park2023robust}, this finite-sample guarantee does not generality extend to misspecified models, i.e., $G \not\in \mathcal{P}_\Theta$. The first result shows that universal inference is prone to be uniformly extremely conservative for working models satisfying QMD and the class of distributions satisfying ULC. We impose the following conditions. Throughout, we consider a non-negative sequence $\{\delta_n\}_{n=1}^\infty$ such that $\delta_n \to 0$ as $n \to \infty$, which is shared across assumptions.
\begin{enumerate}[label=\textbf{(A\arabic*)},leftmargin=2cm]
    \item The working statistical model $\mathcal{P}_{\Theta}$ satisfies QMD for $\{\delta_n\}_{n=1}^\infty$ as in \eqref{eq:dqm}.\label{as:QMD}
    \item \label{as:hessian} Let $V_{G}$ be a positive definite matrix. There exists a  modulus of continuity $\omega_G : \mathbb{R}_+\mapsto \mathbb{R}_+$ such that for any $\delta \le \delta_n$, 
    \begin{align}\label{eq:modulus-hessian}
        \sup_{\|h\|_{V_{G}} \le \delta}\, \left| \E_G\left[\log \frac{p_{\theta_G + h}}{p_{\theta_G}}\big |D_1\right] +\frac{1}{2} h^\top V_{G} h\right| \le \omega_{G} (\delta),
    \end{align}
    and $\omega_{G}(\delta)/\delta^2 \to 0$ as $n \to \infty$. 
     \item Define the tail probability 
     \begin{align*}
         \tau_{n,G}^{\mathrm{C}}(\delta) := \mathbb{P}_G\left(\,\max\left\{\|\widehat\theta_1- \theta_G\|_{I_{G}}, \|\widehat\theta_1- \theta_G\|_{V_{G}}, \frac{\|\widehat\theta_1 - \theta_G\|^2_{V_{G}}}{\|\widehat\theta_1 - \theta_G\|_{I_{G}}}\right\}> \delta \right),
     \end{align*}
     where $C$ stands for consistency, $I_{G}$ and $V_{G}$ are defined in \eqref{eq:dqm} and \eqref{eq:modulus-hessian} respectively. Then, the initial estimator satisfies that 
     \[\limsup_{n\to \infty}\,\sup_{G \in \mathcal{G}}\, \tau_{n,G}^{\mathrm{C}}(\delta_n)= 0.\]
    \label{as:initial_estimator}
\item \label{as:holder_inequality} Let $p \ge 2$ and let $\{L_{n,G}\}_{n=1}^\infty$ be a sequence that may grow with $n$. For any $\delta \le \delta_n$,
\begin{align}\label{eq:dqm-pmoment}
    \sup_{\|h\|_{I_{G}} \le \delta}\,\left(\E_G \left[\left|\frac{p^{1/2}_{\theta_G+h}}{p^{1/2}_{\theta_G}}-1-\frac{1}{2}h^\top \dot\ell_{G}\right|^{p}\right]\right)^{1/p} \le L_{n,G} \delta.
\end{align}
Furthermore, there exists a  modulus of continuity $\nu_{G,p} : \mathbb{R}_+ \mapsto \mathbb{R}_+$ such that for any $\delta \le \delta_n$,
\begin{align*}
        \sup_{\|h\|_{I_{G}} \le \delta}\, \left( \E_G\left[\left|\log \frac{p_{\theta_G + h}}{p_{\theta_G}}\right|^{(2p)/(p-2)}\right]\right)^{(p-2)/(2p)} \le \nu_{G,p} (\delta),
    \end{align*}
and $\max\{L_{n,G}, 1\} \nu_{G,p}(\delta) \to 0$ as $n \to \infty$.
    \item The class of distribution $\mathcal{G}$ satisfies ULC with $H := \dot\ell_{G}(Z)$ as in \eqref{eq:validity-condition-clt}. \label{as:ULC}
\end{enumerate}
Assumptions \ref{as:QMD} and \ref{as:hessian} impose standard regularity conditions on the working statistical model and the behavior of the optimization problem in \eqref{eq:optimization} at the solution $\theta_G$. Specifically, \ref{as:hessian} assumes a local quadratic expansion of the optimization objective around $\theta_G$. This implicitly requires that the derivative of $\theta \mapsto \E_G[\log p_{\theta}]$ vanishes at $\theta_G$ and that the Hessian matrix $\theta \mapsto V_G$ is positive definite and continuous at $\theta = \theta_G$. Assumptions \ref{as:initial_estimator}---\ref{as:ULC} introduce additional conditions specific to the procedure under consideration. Assumption \ref{as:holder_inequality} is introduced since all results in this manuscript are presented under weak smoothness assumptions on $\theta \mapsto \log p_{\theta}$. If Lipschitz continuity of $\log p_{\theta}$ is assumed, i.e., there exists a $G$-measurable function $\eta$ such that
\[|\log p_{\theta+h}(Z)-\log p_{\theta}(Z)| \le \eta(Z)\|h\|\]
then \ref{as:holder_inequality} roughly amounts to requiring a finite $2p/(p-2)$ moment of $\eta$. This type of assumption is not uncommon in the literature (See Section 3 of \citet{wasserman2020universal}, A1 of \citet{strieder2022choice}, M1 of \citet{shao2022berry}, for instance). Under the Lipschitz continuity assumption, however, the proof becomes more straightforward, and \ref{as:holder_inequality} can be removed completely as the log-likelihood ratio admits an LAN expansion only assuming \ref{as:QMD} and \ref{as:hessian}. See, for instance, Lemma 2.1 of \citet{Kleijn2012Bernstein}. Assumption \ref{as:initial_estimator} can also be dropped but \ref{as:ULC} will become more complicated (See Remark~\ref{remark:dimension-smoothness}). Finally, assumptions \ref{as:QMD}---\ref{as:ULC} are sufficient conditions for uniformly controlling the Berry-Esseen bound over the class of distributions $\mathcal{G}$, though it is not strictly necessary, as discussed in \citet{takatsu2025bridging}. 

In the remainder of the results, we frequently employ the notation $\mathbb{G}^\times_n$ to denote the conditional empirical process. For any $G$-measurable function $f$, we define
\begin{align}\label{eq:cond-Gn}
    \mathbb{G}^\times_n f := n^{-1/2}\sum_{i\in D_2}\left(f(Z_i)-\E_G[f(Z)|D_1]\right).
\end{align}
\begin{theorem}\label{thm:ui-conservatism}
Assume \ref{as:QMD}, \ref{as:hessian} and \ref{as:holder_inequality}. Define the Kolmogorov-Smirnov distance for $G \in \mathcal{G}$ as
\begin{align*}
    \Delta_{n, G}^{\mathrm{BE}} := \sup_{t\in\mathbb{R}}\left|\mathbb{P}_G(\mathbb{G}^\times_n [\sigma^{-1}_{\theta_G, \widehat\theta_1}\xi
    ] \le t | D_1) - \Phi(t)\right|
\end{align*}
where $\xi = \log p_{\widehat\theta_1}(Z) - \log p_{\theta_G}(Z)$ and $\sigma^2_{\theta_G, \widehat\theta_1} = \mathrm{Var}_{G}[\xi | D_1]$. 
    The nonasymptotic miscoverage probability of the universal confidence set for any $G$ satisfies 
    \begin{align*}
        &|\,\mathbb{P}_{G}(\theta_G \not\in \mathrm{CI}^{\mathrm{UI}}_{n, \alpha}) - \big(1-\E_G[\Phi(z_\alpha + r^{\mathrm{UI}}_{n,G}) \big)]\, |  \\
        &\qquad \le \min\bigg\{1, \E_G[\Delta_{n, G}^{\mathrm{BE}}] + \inf_{\delta \in (0, \delta_n]}\left\{\mathcal{R}^{\mathrm{UI}}_{n,G}(\delta) + \tau_{n,G}^{\mathrm{C}}(\delta)\right\}\bigg\}
    \end{align*} 
    where
    \begin{align*}
        r^{\mathrm{UI}}_{n,G} :=  \frac{\log(1/\alpha)}{n^{1/2}\|\widehat\theta_1 - \theta_G\|_{I_{G}}} + \frac{n^{1/2}\|\widehat\theta_1 - \theta_G\|^2_{V_{G}}}{2\|\widehat\theta_1 - \theta_G\|_{I_{G}}} - z_\alpha,
    \end{align*}
    and
    \begin{align*}
        &\mathcal{R}^{\mathrm{UI}}_{n,G}(\delta) :=  C_{p}\left\{\delta^2 + \left(\frac{\phi_G(\delta)}{\delta^2}\right)^{1/2} + \frac{\omega_{G}(\delta)}{\delta^2} + \left(\max\left\{L_{n,G}, 1\right\}\nu_{G,p}(\delta)\right)^{1/2}\right\}
    \end{align*}
    with a constant $C_p$ only depending on $p$ in \ref{as:holder_inequality}. Additionally, suppose that the initial estimator $\widehat\theta_1$ satisfies \ref{as:initial_estimator} and the class of distributions $\mathcal{G}$ satisfies \ref{as:ULC}. Then   
    \begin{align*}
        \limsup_{n\to \infty}\, \sup_{G\in\mathcal{G}}\, |\,\mathbb{P}_{G}(\theta_G \not\in \mathrm{CI}^{\mathrm{UI}}_{n, \alpha}) - \big(1-\E_G[\Phi(z_\alpha + r^{\mathrm{UI}}_{n,G})]\big)\, | = 0.
    \end{align*}
\end{theorem}
The proof of Theorem~\ref{thm:ui-conservatism} is provided in Section~\ref{supp:proof-ui} of the supplement. Theorem~\ref{thm:ui-conservatism} reveals the asymptotic behavior of the universal confidence set across different regimes, depending on the random variable $r^{\mathrm{UI}}_{n,G}$: (1) when $ r^{\mathrm{UI}}_{n,G}\overset{d}{\to}\infty$, the procedure becomes extremely conservative; (2) when $ r^{\mathrm{UI}}_{n,G}\overset{d}{\to} 0$, the procedure becomes exact; (3) when $r^{\mathrm{UI}}_{n,G}$ is negative, the procedure becomes invalid. We claim that (1) is pervasive for high-dimensional or nonparametric problems, while (2) requires some ``miracle". The following corollaries explicitly characterize these cases.
\begin{corollary}\label{cor:trivial-conservative}
    Assume \ref{as:QMD}---\ref{as:ULC}. The universal confidence set is uniformly extremely conservative as in Definition~\ref{def:uniformly-conservative} (and thus valid) if for any $\varepsilon > 0$
     \begin{align}\label{eq:either-or}
     \limsup_{n\to\infty}\sup_{G \in \mathcal{G}}\,\mathbb{P}_G\left(\min\left\{n^{1/2}\|\widehat \theta_1 - \theta_G\|_{I_{G}}, \frac{\|\widehat \theta_1 - \theta_G\|_{I_{G}}}{n^{1/2}\|\widehat \theta_1 - \theta_G\|^2_{V_{G}}}\right\} > \varepsilon \right) = 0.
    \end{align}
\end{corollary}
\begin{corollary}\label{cor:trivial-correct}
    Assume \ref{as:QMD}---\ref{as:ULC}. Then for any $G \in \mathcal{G}$, 
    \begin{align*}
        \Phi(z_\alpha + r^{\mathrm{UI}}_{n,G})  \ge \Phi(\sqrt{2\log(1/\alpha)\mathfrak{R}_G(\widehat\theta_1 - \theta_G)} )   \quad \text{where}\quad\mathfrak{R}_G(h) = \|h\|_{ V_{G}}^2/\|h\|_{ I_{G}}^2.
    \end{align*}
    In particular, whenever
    \begin{align}\label{eq:invalid-condition}
        \inf_{G\in\mathcal{G}}\sqrt{2\log(1/\alpha) \lambda_{\min}(I^{-1/2}_{G}V_{G} I^{-1/2}_{G})}> z_\alpha,
    \end{align}
    universal inference is provably uniformly conservative, i.e.,
    \begin{align*}
        \liminf_{n\to \infty}\,  \inf_{G\in\mathcal{G}}\,\mathbb{P}_{G}(\theta_G \in \mathrm{CI}^{\mathrm{UI}}_{n, \alpha})  > 1-\alpha.
    \end{align*}
    When $G \in \mathcal{P}_{\Theta}$, it holds that $I_{G} = V_{G}$, which implies 
    \begin{align}\label{eq:ui-correct-model}
        \liminf_{n\to \infty}\,  \inf_{G\in\mathcal{G} \cap\mathcal{P}_{\Theta}}\,\mathbb{P}_{G}(\theta_G \in \mathrm{CI}^{\mathrm{UI}}_{n, \alpha})  \ge \Phi(\sqrt{2\log(1/\alpha)} )  > 1-\alpha
    \end{align}
    for any $\alpha$ bounded away from zero or one. Thus, universal inference can never be exact when the model is correctly specified.
\end{corollary}

The proofs of Corollaries~\ref{cor:trivial-conservative} and \ref{cor:trivial-correct} are provided in Section~\ref{supp:proof-cor} of the supplement. One of the main claims of \citet{wasserman2020universal} is that finite-sample validity holds regardless of the choice of the initial estimator, as long as the model is correctly specified. However, there is no contradiction in stating that the asymptotic conservativeness crucially depends on the properties of the initial estimator.

Corollary~\ref{cor:trivial-conservative} establishes that universal inference becomes extremely conservative when the initial estimator converges either strictly faster or slower than the root-$n$ rate. In particular, when the dimension of $\theta_G$ is large or $\theta_G$ belongs to a nonparametric class, the method becomes extremely conservative. This aligns with the empirical findings of \citet{tse2022note}. Interestingly, it also implies that universal inference remains asymptotically valid under model misspecification in high-dimensional settings. This does not contradict the findings of \citet{park2023robust} since they focus on the finite-sample guarantees. Indeed, the empirical illustration in Section~\ref{sec:illusration} reveals that while misspecified universal inference can be invalid when the dimension of $\theta_G$ is negligible compared to the sample size, it becomes extremely conservative as the dimension becomes comparable to the sample size (See the right panel of Figure~\ref{fig:coverage}).

Corollaries~\ref{cor:trivial-conservative} and \ref{cor:trivial-correct} together imply that universal inference is uniformly valid when the model is correctly specified, consistent with the original result of \citet{wasserman2020universal}. However, Corollary~\ref{cor:trivial-correct} highlights that universal inference will never achieve the exact coverage level of $1-\alpha$ even under the correct model. Specifically, it cannot achieve the coverage lower than $\Phi(\sqrt{2\log(1/\alpha}))$, up to the random fluctuations from $\Delta_{n, G}^{\mathrm{BE}}$ and the other remainder terms. For instance, when $\alpha=0.05$, the coverage of universal inference will be approximately $0.9928$ or larger. Empirically, this bound provides an extremely accurate estimate of the coverage of universal inference, as demonstrated by the dash-dotted line in the left panel of Figure~\ref{fig:coverage}.

Under condition \eqref{eq:invalid-condition}, universal inference with a misspecified model is provably conservative, and thus valid. More interestingly, if not more alarmingly, universal inference may become invalid, exact or extremely conservative when the model is misspecified. Section~\ref{supp:invalid} of the supplement provides an informal discussion of when the universal confidence set may become invalid. Specifically, when 
\[\sqrt{2\log(1/\alpha) \lambda_{\max}(I^{-1/2}_{G}V_{G} I^{-1/2}_{G})} <  z_\alpha,\]
it is \emph{possible} for the universal confidence set to be invalid, though this is not the only scenario in which universal inference may fail. More broadly, under model misspecification, universal inference lacks an interpretable validity statement, as it may be either conservative or invalid, depending on the behaviors of $\mathfrak{R}_G$ and the initial estimator. The term $\mathfrak{R}_G$ in Corollary~\ref{cor:trivial-correct} represents the ratio between the sandwich variance of the MLE and its variance under the correct model, projected onto some random vector. \citet{Buja2019Models} refers to this quantity as the \emph{RAV} (Ratio of Asymptotic Variances) in the context of assumption-lean linear regression. Section 11 of \citet{Buja2019Models} discusses factors influencing whether RAV is large or small, such as the presence of heterogeneity near the tail of the data-generating distribution. These characteristics will play a role in determining the validity of misspecified universal inference.

The following subsection explores potential solutions to address these challenges in universal inference.

\subsection{Studentized and Bias-corrected Universal Inference}
This section proposes modifications to universal inference through studentization and bias correction. The construction follows \citet{takatsu2025bridging}, which introduced a general confidence set for M-estimation problems. We apply their framework to the KL-projection problem in \eqref{eq:optimization}. 

The key idea is as follows. First, observe that the universal confidence set, as defined in \eqref{eq:ui-def}, involves the sum of independent random variables. Consequently, the CLT may apply when the summands are scaled by the sample standard deviation. Although this approach does not account for the fact that the summands are uncentered, the limiting distribution is always downward biased due to the definition of the optimization problem in \eqref{eq:optimization}. This property allows for the construction of a conservative confidence set. The resulting confidence set, obtained through studentization, is given by
 \begin{align}\label{eq:std-ui-def}
     \mathrm{CI}^{\mathrm{Std}}_{n, \alpha} := \left\{\theta \in \Theta \, : \, n^{-1}\sum_{i \in D_2}\, \log \frac{p_{\widehat\theta_1}(Z_i)}{p_{\theta}(Z_i)} \le n^{-1/2}z_{\alpha}\widehat\sigma_{\theta, \widehat\theta_1} \right\} \quad \text{with}\quad n = |D_2|
 \end{align}
 where $\widehat\sigma^2_{\theta, \widehat\theta_1}$ is a sample variance of $\{\log p_{\widehat \theta_1}(Z_i)-\log p_\theta(Z_i) \, : \, i \in D_2\}$. A similar method, termed \emph{KL Relative Divergence Fit}, was proposed by \citet{park2023robust}, except that they added an independent Gaussian noise to each observation $\log p_{\widehat \theta_1}(Z_i)-\log p_\theta(Z_i)$ for theoretical purposes. This was later proven unnecessary under ULC by \citet{takatsu2025bridging}.
 
While the unaddressed downward bias does not affect validity, it can significantly impact the conservativeness of the procedure. For the KL-projection defined as \eqref{eq:optimization}, the bias at the functional $\theta_G$ is characterized as 
\begin{align*}
    \left|\E_G\left[\log \frac{p_{\widehat\theta_1}}{p_{\theta_G}}\bigg|D_1 \right]
    +\frac{1}{2} (\widehat \theta_1 - \theta_G)^\top  V_{G} (\widehat \theta_1 - \theta_G) \right| \le \omega_{G}(\delta) \quad \text{provided} \quad \|\widehat \theta_1 - \theta_G\|_{V_{G}} \le \delta
\end{align*}
for some modulus of continuity $\omega_{G}$ in view of \ref{as:hessian}. Let $\widehat V_{G}$ be an estimator of the Hessian matrix $V_{G}$, constructed from $\{Z_i \, : \, i\in D_2\}$. See the left panel of Figure~\ref{fig:sample-splitting} for illustration. Once this estimator is available, we can construct a bias-corrected confidence set by manually ``subtracting" the estimated bias:
 \begin{align*}
     \mathrm{CI}^{\mathrm{BC}}_{n, \alpha} := \left\{\theta \in \Theta \, : \, n^{-1}\sum_{i \in D_2}\, \log \frac{p_{\widehat\theta_1}(Z_i)}{p_{\theta}(Z_i)}+\frac{1}{2} (\widehat \theta_1 - \theta)^\top  \widehat V_{G} (\widehat \theta_1 - \theta) \le n^{-1/2}z_{\alpha} \widehat\sigma_{\theta, \widehat\theta_1}\right\}.
 \end{align*}
 This approach was also mentioned in \citet{takatsu2025bridging}. We refer to the final confidence set $\mathrm{CI}^{\mathrm{BC}}_{n, \alpha}$ as \emph{studentized and bias-corrected universal inference}. Before providing the miscoverage result, we introduce the following additional conditions:
 \begin{enumerate}[label=\textbf{(A\arabic*)},leftmargin=2cm]
\setcounter{enumi}{5}
     \item  \label{as:fourth-moment} Let $\{\delta_n\}_{n=1}^\infty$ be the same sequence as in \ref{as:QMD}---\ref{as:holder_inequality}. There exists a  modulus of continuity $\nu_{G,4} : \mathbb{R}_+ \mapsto \mathbb{R}_+$ such that for any $\delta < \delta_n$
\begin{align*}
        \sup_{\|h\|_{I_{G}} \le \delta}\, \left( \E_G\left[\left(\log \frac{p_{\theta_G + h}}{p_{\theta_G}}\right)^{4}\right]\right)^{1/4} \le \nu_{G,4} (\delta),
    \end{align*}
and $n^{-1} (\nu_{G,4}(\delta)/\delta)^4 \to 0$ as $n \to \infty$.
    \item \label{as:double-robust} Let $u = I^{1/2}_{\theta_G}(\widehat \theta_1 - \theta_G) /\|\widehat \theta_1 - \theta_G\|_{I_{G}} \in \mathbb{S}^{d-1}$. Define the tail probability, 
    \begin{align*}
        \tau_{n,G}^{\mathrm{DR}}(\gamma) = \mathbb{P}_G\left(\,n^{1/2}\|\widehat \theta_1 - \theta_G\|_{I_{G}}|u^\top I^{-1/2}_{G}(\widehat V_{G}-V_{G})I^{-1/2}_{G}u| > \gamma \right),
    \end{align*}
    where DR stands for double robustness. There exists a non-negative sequence $\{\gamma_n\}_{n=1}^\infty$ such that $\gamma_n \to 0$ as $n \to \infty$ and
    \begin{align}
    \limsup_{n\to\infty}\,\sup_{G\in\mathcal{G}}\,\tau_{n,G}^{\mathrm{DR}}(\gamma_n)= 0.
    \end{align}
\end{enumerate}
Assumption \ref{as:fourth-moment} ensures that the sample variance consistently estimates its population counterpart. Assumption \ref{as:double-robust} establishes a form of \emph{double robustness} between the estimators $\widehat \theta_1$ and $u^\top \widehat V_{G}u$, requiring their errors to satisfy a second-order bias condition. Although $\widehat V_{G}$ is a $d \times d$ matrix, \ref{as:double-robust} is only stated through a \emph{fixed} univariate projection, effectively reducing the estimation problem to a scalar component. This reduction is a crucial consequence of estimating $\widehat \theta_1$ and $\widehat V_{G}$ from independent sets of data. For further discussion, see Remark~\ref{remark:sample-splitting}. We now present the miscoverage result for the proposed studentized and bias-corrected universal inference.
\begin{theorem}\label{thm:bc-conservatism}
Assume \ref{as:QMD}, \ref{as:hessian}, \ref{as:holder_inequality} and \ref{as:fourth-moment}. For any $G$, the nonasymptotic miscoverage probability of the confidence set satisfies 
\begin{align*}
        &|\,\mathbb{P}_{G}(\theta_G \not\in \mathrm{CI}^{\mathrm{BC}}_{n, \alpha}) - \alpha\, |  \\
        &\qquad \le \min\bigg\{1, \E_G[\Delta_{n, G}^{\mathrm{BE}}] + \inf_{\delta \in (0, \delta_n]}\left\{ \left(\mathcal{R}^{\mathrm{BC}}_{n,G}(\delta)\right)^{1/2}+ \tau_{n,G}^{\mathrm{C}}(\delta)\right\}+\inf_{\gamma \in (0, \gamma_n]}\left\{\gamma + \tau_{n,G}^{\mathrm{DR}}(\gamma)\right\}\bigg\}
    \end{align*} 
and 
    \begin{align*}
        &\mathcal{R}^{\mathrm{BC}}_{n,G}(\delta) := C_{p}\left\{\delta^2 + \left(\frac{\phi_G(\delta)}{\delta^2}\right)^{1/2} + \frac{\omega_{G}(\delta)}{\delta^2} + \left(\max\left\{L_{n,G}, 1\right\}\nu_{G,p}(\delta)\right)^{1/2} + \frac{\nu_{G,4}^2(\delta)}{n^{1/2}\delta^2}\right\}
    \end{align*}
    with a constant $C_p$ only depending on $p$ in \ref{as:holder_inequality}.
    Additionally, suppose that the initial estimator $\widehat\theta_1$ satisfies \ref{as:initial_estimator} and \ref{as:double-robust} and the class of distributions $\mathcal{G}$ satisfies \ref{as:ULC}. Then   
    \begin{align*}
        \limsup_{n\to \infty}\, \sup_{G\in\mathcal{G}}\, |\,\mathbb{P}_{G}(\theta_G \not\in \mathrm{CI}^{\mathrm{BC}}_{n, \alpha}) - \alpha \, | = 0.
    \end{align*}
\end{theorem}

The proof of Theorem~\ref{thm:bc-conservatism} is provided in Section~\ref{supp:proof-bc} of the supplement. Theorem~\ref{thm:bc-conservatism} demonstrates a significant improvement over the original universal inference in terms of asymptotic conservativeness. The proposed method, based on studentization and bias correction, achieves exact coverage once the second-order bias condition is satisfied. Heuristically, for a $d$-dimensional inferential problem, this requirement can be interpreted as
\[n^{1/2}\|\widehat \theta_1-\theta_G\|_{I_{G}} u^\top  I_{G}^{-1/2}(\widehat V_{G} - V_{G}) I_{G}^{-1/2}u = n^{1/2}\cdot O_P\left(\sqrt{\frac{d}{n}}\right)\cdot O_P(n^{-1/2})= O_P\left(\sqrt{\frac{d}{n}}\right).\]
Therefore, the proposed confidence set is expected to achieve near $1-\alpha$ coverage probability when $d = o(n)$. Finally, Theorem~\ref{thm:bc-conservatism} does not rule out the possibility where the  studentized and bias-corrected confidence set becomes invalid for some choice of the estimator $\widehat V_G$, particularly when \ref{as:double-robust} is violated.

\begin{remark}[On the estimation of $V_{G}$]\label{remark:sample-splitting}
The proposed method assumes that the 
initial estimator $\widehat\theta_1$ is computed from $\{Z_i\, : \, i \in D_1\}$ while $\widehat V_{G}$ is obtained from $\{Z_i\, : \, i \in D_2\}$ (See the left panel of Figure~\ref{fig:sample-splitting}). As a result, \ref{as:double-robust} is stated only for univariate projection, benefiting from the independence between the two estimators. When $\widehat \theta_1$ and $\widehat V_{G}$ are obtained from the same data (See the right panel of Figure~\ref{fig:sample-splitting}), however, the requirement becomes 
    \[n^{1/2}\|\widehat \theta_1-\theta_G\|_{I_{G}} \|I^{-1/2}_{G}(\widehat V_{G}-V_{G})I^{-1/2}_{G}\|_{\mathrm{op}} = o_P(1) \quad \text{for} \quad G \in \mathcal{G}.\]
Since it involves estimating a $d \times d$ matrix in the operator norm, some form of growth condition on $d$ is likely unavoidable unless additional structural assumptions, such as sparsity, are imposed. We anticipate that bias correction techniques, such as those in \citet{chang2023inference}, may help mitigate this requirement. Mirroring the results in \citet{chang2023inference}, we conjecture that extreme conservativeness can be avoided when $d = o(n^{2/3})$.

Alternatively, one may consider splitting data into three parts, each dedicated to the construction of $\widehat \theta_1$, $\widehat V_{G}$, and $\mathrm{CI}^{\mathrm{BC}}_{n, \alpha}$ (See the right panel of Figure~\ref{fig:sample-splitting-2} in the supplement). Interestingly, this can lead to an invalid confidence set for large $d$ while the proposed data-splitting scheme, which uses the same data for the construction of $\widehat V_{G}$ and $\mathrm{CI}^{\mathrm{BC}}_{n, \alpha}$, remains valid for the specific numerical problem we analyze in Section~\ref{sec:illusration}. See Section~\ref{supp:self-center} of the supplement for more discussion.
\end{remark}

\begin{figure}
\centering
\begin{subfigure}{0.9\textwidth}
  \centering
  \includegraphics[width=\linewidth]{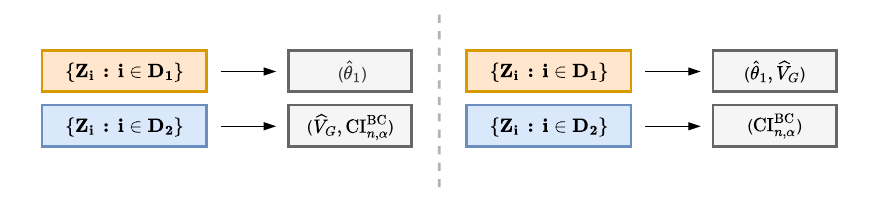}
\end{subfigure}
\caption{A schematic illustrating the sample splitting procedure. The arrow indicates that the objects are derived from the corresponding data. The left panel shows the proposed procedure, while the right panel provides the illustration used in Remark~\ref{remark:sample-splitting}.}
\label{fig:sample-splitting}
\end{figure}
\begin{remark}[Studentization without bias correction]\label{remark:std-only}
In certain scenarios, such as when \ref{as:hessian} is violated due to the optimizer lying on the boundary, or in highly non-smooth problems where the parameter space $\Theta$ is discontinuous, such as \cite{zhang2024winners}, obtaining an estimator of $V_G$ can be challenging or even infeasible. Section~\ref{supp:std} of the supplement provides an analogous result to Theorems~\ref{thm:ui-conservatism} and \ref{thm:bc-conservatism} for universal inference only using studentization. However, the result proves to be unfavorable, as it leads to the bound
   \begin{align*}
        \limsup_{n\to \infty}\,\sup_{G\in \mathcal{G}}\,|\,\mathbb{P}_{G}(\theta_G \not\in \mathrm{CI}^{\mathrm{Std}}_{n, \alpha}) - \big(1-\E_G[\Phi(z_\alpha + r^{\mathrm{Std}}_{n,G})] \big) = 0
    \end{align*} 
    where
    \begin{align}
        r^{\mathrm{Std}}_{n,G} =2^{-1}n^{1/2}\|\widehat \theta_1 - \theta_G\|_{I_{G}}u^\top(I^{-1/2}_{G}V_{G} I^{-1/2}_{G})u
    \end{align}
    with $u = I^{1/2}_{\theta_G}(\widehat \theta_1 - \theta_G) /\|\widehat \theta_1 - \theta_G\|_{I_{G}} \in \mathbb{S}^{d-1}$. This implies that, when the initial estimator converges at a rate strictly slower than root-$n$, relying solely on studentization leads to an extremely conservative confidence set since $r^{\mathrm{Std}}_{n,G} \overset{d}{\to} \infty$. Nevertheless, studentization provides a key benefit to universal inference. Since $r^{\mathrm{Std}}_{n,G} \ge 0$ by definition, the resulting confidence set remains valid under model misspecification. This result should be compared to Theorem~\ref{thm:bc-conservatism}, which requires \ref{as:double-robust} for the validity. This robustness is a significant improvement over universal inference without studentization. 
\end{remark}
\begin{remark}[General power analysis of universal confidence set]\label{remark:power-ui} To the best of our knowledge, a comprehensive width analysis for universal confidence sets under model misspecification is not available in the literature. The closest result is Theorem 5 of \citet{park2023robust}, which provides the shrinkage rate of their confidence set procedures. However, their results are currently limited to cases where the log-likelihood function is uniformly bounded. In contrast, Theorem 8 of \citet{takatsu2025bridging} can be applied to derive a convergence rate for the proposed studentized and bias-corrected universal inference, which allows for unbounded, and possibly heavy-tailed, log-likelihood functions. 
\end{remark}

\begin{remark}[Implicit dimension dependence]\label{remark:dimension-smoothness}
Throughout the manuscript, we have implicitly assumed the existence of a consistent estimator as \ref{as:QMD}, \ref{as:hessian}, \ref{as:holder_inequality} and \ref{as:fourth-moment} provide the corresponding condition for $\theta_G + h$ as $\|h\| \to 0$. This requires conditions on the growth of the dimension such as $d=o(n)$ or additional assumptions like sparsity. Remark 8 of \citet{takatsu2025bridging} suggests that the validity of the confidence set can be established without the existence of consistent estimators, though this would require an alternative assumption to \ref{as:ULC}, that are more challenging to interpret. 
\end{remark}

\section{Illustration --- Misspecified Linear Regression}\label{sec:illusration}

This section illustrates the role of studentization and bias correction in addressing conservativeness in high-dimensional MLE, using a misspecified linear regression model as an example. Consider IID observations $(Y_1, X_1^\top)^\top, \ldots, (Y_n, X_n^\top)^\top \in \mathbb{R} \times \mathbb{R}^d$, generated from the model
\begin{align} \label{eq:model}
Y_i = \theta_G^\top X_i + \epsilon_i, \quad \text{where} \quad \mathbb{E}_G[\epsilon_i X_i] = 0, \quad \text{and} \quad \mathbb{E}_G[\epsilon_i^2 | X_i] = \sigma_i^2.
\end{align}
We assume that the Gram matrix $\mathbb{E}_G[XX^\top]$ is invertible, ensuring the existence of $\theta_G \in \mathbb{R}^d$, even when the true regression function $\mathbb{E}_G[Y_i | X_i]$ is not linear. The target parameter $\theta_G$ can be expressed as the solution to the following optimization problem
\begin{equation}\label{eq:least-square}
\theta_G := \argmin_{\theta \in \mathbb{R}^d} \mathbb{E}_G[(Y - \theta^\top X)^2].
\end{equation}
Suppose universal confidence sets are constructed under a generally misspecified parametric model:
\begin{align*}
\log p_{\theta}(Y_i, X_i) = -\frac{1}{2} \log (2\pi\sigma^2) - \frac{1}{2\sigma^2} (Y_i - \theta^\top X_i)^2,
\end{align*}
where $\sigma \in (0, \infty)$ is assumed to be fixed. In this case, the MLE for $\theta$ under misspecification coincides with the ordinary least squares (OLS) estimator. However, the behavior of universal inference varies significantly depending on the choice of $\sigma$. Suppose the data is split into two parts: \[D_1 := \{(X_i, Y_i) \, : \, 1\le i \le \lfloor N/2\rfloor \} \text{ and } D_2 := \{(X_i, Y_i) \, : \, \lfloor N/2\rfloor + 1 \le i \le N\}.\] Let $\widehat \theta_1$ denote the OLS estimator (or equivalently the MLE) computed from $D_1$, and let $n = |D_2|$. We consider two instantiations of universal inference:
\begin{itemize}
    \item When $\sigma = 1:$
    \begin{align}\label{eq:ui-def-model}
     \mathrm{CI}^{\mathrm{UI-1}}_{n, \alpha} = \left\{\theta \in \Theta \, : \, \sum_{i \in D_2}\,(Y_i-\theta^\top X_i)^2 -(Y_i-\widehat\theta_1^\top X_i)^2  \le 2\log(1/\alpha) \right\}.
 \end{align}
 \item When $\sigma = 0.1:$ \begin{align}\label{eq:ui-def-incorrect}
     \mathrm{CI}^{\mathrm{UI-2}}_{n, \alpha} = \left\{\theta \in \Theta \, : \, \sum_{i \in D_2}\,(Y_i-\theta^\top X_i)^2 -(Y_i-\widehat\theta_1^\top X_i)^2 \le (2\cdot 0.01)\log(1/\alpha) \right\}.
 \end{align}
\end{itemize}

Next, we define the studentized and bias-corrected universal inference for this problem. Noting that studentization is scale-invariant, we can multiply both sides of the inequality in the definition \eqref{eq:std-ui-def} by $2\sigma^2$ to adjust for the scale. This leads to the following studentized universal confidence set: 
\begin{align}\label{eq:std-ui-lr}
     \mathrm{CI}^{\mathrm{Std}}_{n, \alpha} := \left\{\theta \in \Theta \, : \, n^{-1}\sum_{i \in D_2}\, (Y_i-\theta^\top X_i)^2 -(Y_i-\widehat\theta_1^\top X_i)^2 \le n^{-1/2}z_{\alpha}\widehat\sigma_{\theta, \widehat\theta_1} \right\},
\end{align}
where $\widehat\sigma^2_{\theta, \widehat\theta_1}$ is the sample variance of $\{(Y_i-\theta^\top X_i)^2 -(Y_i-\widehat\theta_1^\top X_i)^2 \, : \, i \in D_2\}$. Finally, under the model assumption \eqref{eq:model}, without imposing linearity, we obtain
\begin{align*}
    -\E_G \left[\log \frac{p_{\widehat\theta_1}}{p_{\theta_G}} \bigg|D_1\right]&= -\frac{1}{2 \sigma^2} \left\{\E_G[(Y_i-\theta_G^\top X_i)^2] -\E_G[(Y_i-\widehat\theta_1^\top X_i)^2|D_1]\right\}\\
    &=\frac{1}{2 \sigma^2} (\widehat\theta_1 -\theta_G)^\top \E_G[XX^\top](\widehat\theta_1 -\theta_G).
\end{align*}
To estimate the Gram matrix, we use the sample Gram matrix:
\begin{align*}
     (\widehat\theta_1 -\theta_G)^\top\widehat V_{G}(\widehat\theta_1 -\theta_G) = \frac{1}{2\sigma^2}(\widehat\theta_1 -\theta_G)^\top \left(n^{-1}\sum_{i \in D_2} X_i X_i^\top\right)(\widehat\theta_1 -\theta_G).
\end{align*}
After scaling both sides by $2\sigma^2$, the studentized and bias-corrected universal confidence set is defined as
\begin{equation}\label{eq:bc-ui-lr}
    \begin{aligned}
        \mathrm{CI}^{\mathrm{BC}}_{n, \alpha} := \bigg\{\theta \in \Theta \, : \, &n^{-1}\sum_{i \in D_2}\, (Y_i-\theta^\top X_i)^2 -(Y_i-\widehat\theta_1^\top X_i)^2 \\
&\qquad+ (\widehat \theta_1 - \theta)^\top  \bigg(n^{-1} \sum_{i \in D_2} X_iX_i^\top\bigg)(\widehat \theta_1 - \theta) \le n^{-1/2}z_{\alpha} \widehat\sigma_{\theta, \widehat\theta_1}\bigg\}.
    \end{aligned}
\end{equation}
The confidence sets in \eqref{eq:std-ui-lr} and \eqref{eq:bc-ui-lr} coincide with the confidence set for the least squares problem \eqref{eq:least-square}, as proposed by \citet{takatsu2025bridging}. Theorem 12 of \citet{takatsu2025bridging} already establishes that the diameter of $\mathrm{CI}^{\mathrm{Std}}_{n, \alpha}$ converges at a minimax rate. Since $\mathrm{CI}^{\mathrm{BC}}_{n, \alpha} \subseteq \mathrm{CI}^{\mathrm{Std}}_{n, \alpha}$, the diameter of $\mathrm{CI}^{\mathrm{BC}}_{n, \alpha}$ also converges at a minimax rate. 

We now assess the empirical coverage of four different confidence sets. For given $N \ge 2$ and $d$, independent observations $(X_i, Y_i) \in \mathbb{R}^d \times \mathbb{R}$, $1 \le i \le N$ are generated from the following model:
\begin{align}\label{eq:exp-model1}
    Y_i = \theta_G^\top X_i + \epsilon_i  \quad\text{where}\quad X_i \sim N(0_d, I_d) \quad \text{and}\quad \epsilon_i \sim N(0,1),
\end{align}
with $\theta_G = (d^{-1/2},\ldots, d^{-1/2})$. This setup makes the first universal confidence set, given by \eqref{eq:ui-def-model}, correctly specified, which is seemingly the most favorable scenario for this method. In contrast, the second universal confidence set, given by \eqref{eq:ui-def-incorrect}, is misspecified due to the incorrect choice of the variance term. It is straightforward to verify \ref{as:QMD}---\ref{as:fourth-moment}. Assumption \ref{as:double-robust} is satisfied when $d = o(n)$.

Figure~\ref{fig:coverage} displays the empirical coverage of the 95\% confidence sets, computed from 1000 replications. The four methods under consideration are labeled as follows: \emph{Universal Inference (Correct)} corresponds to the confidence set in \eqref{eq:ui-def-model}, \emph{Universal Inference (Incorrect)} to the confidence set in \eqref{eq:ui-def-incorrect}, \emph{Studentized} to the confidence set in \eqref{eq:std-ui-lr}, and \emph{Studentized + Bias-corrected} to the confidence set in \eqref{eq:bc-ui-lr}. The nominal level of $0.95$ is shown as a dashed line. Additionally, we provide the theoretical lower bound coverage level of universal inference under the correct model, given by $\Phi(\sqrt{2\log(1/\alpha)})$ with $\alpha=0.05$, shown as a dash-dotted line in the left panel. See the inequality~\eqref{eq:ui-correct-model} in Corollary~\ref{cor:trivial-correct}.

The left panel of Figure~\ref{fig:coverage} displays the empirical coverage for each method when $d=5$ and $10 \le N \le 5000$. In this scenario, we observe that both \emph{Universal Inference (Correct)} and \emph{Studentized} are conservative, though not extremely so, with their coverage probabilities slightly below 1. This behavior is consistent with Corollary~\ref{cor:trivial-correct} and Remark~\ref{remark:std-only}. Empirically, the coverage probability aligns closely with the theoretical lower bound, $\Phi(\sqrt{2\log(1/\alpha)})\approx 0.9928$. On the other hand, \emph{Universal Inference (Incorrect)} remains invalid even as the sample size increases. Finally, \emph{Studentized + Bias-corrected} achieves coverage accuracy near the $1-\alpha$ level, in agreement with Theorem~\ref{thm:bc-conservatism}, as we are in the regime where $d = o(n)$.

The right panel of Figure~\ref{fig:coverage} displays the coverage for each method when $N=500$ and $2 \le d \le 250$ (with $n = |D_2| = 250$). Here, both \emph{Universal Inference (Correct)} and \emph{Studentized} become extremely conservative as the dimension $d$ increases, as expected from Corollary~\ref{cor:trivial-conservative} and Remark~\ref{remark:std-only}. \emph{Universal Inference (Incorrect)} is invalid for small values of $d$, but eventually becomes extremely conservative as $d$ increases, confirming Corollary \ref{cor:trivial-conservative}. Finally, \emph{Studentized + Bias-corrected} maintains $1-\alpha$ level coverage when $d = o(n)$, gradually becoming conservative, and ultimately reaching extremely conservative as $d$ approaches $n$. This agrees with Theorem~\ref{thm:bc-conservatism} in the regime where $d=o(n)$; however, the theorem does not provide the validity guarantee beyond this regime. Interestingly, the confidence set remains valid even when $d$ is comparable to $n$. Section~\ref{supp:self-center} of the supplement provides an informal explanation to this phenomenon. The discussion in Section~\ref{supp:self-center} suggests that the use of the same dataset to estimate $V_G$ and to construct the confidence set leads to an implicit bias correction through self-centering. While the effect of model misspecification in this section appears mild, as the observations are generated from a Gaussian model, Section~\ref{sup:laplace} of the supplement provides additional results under more prominent model misspecification, where the general narratives remain similar.
\begin{figure}
\centering
\begin{subfigure}{0.9\textwidth}
  \centering
  \includegraphics[width=\linewidth]{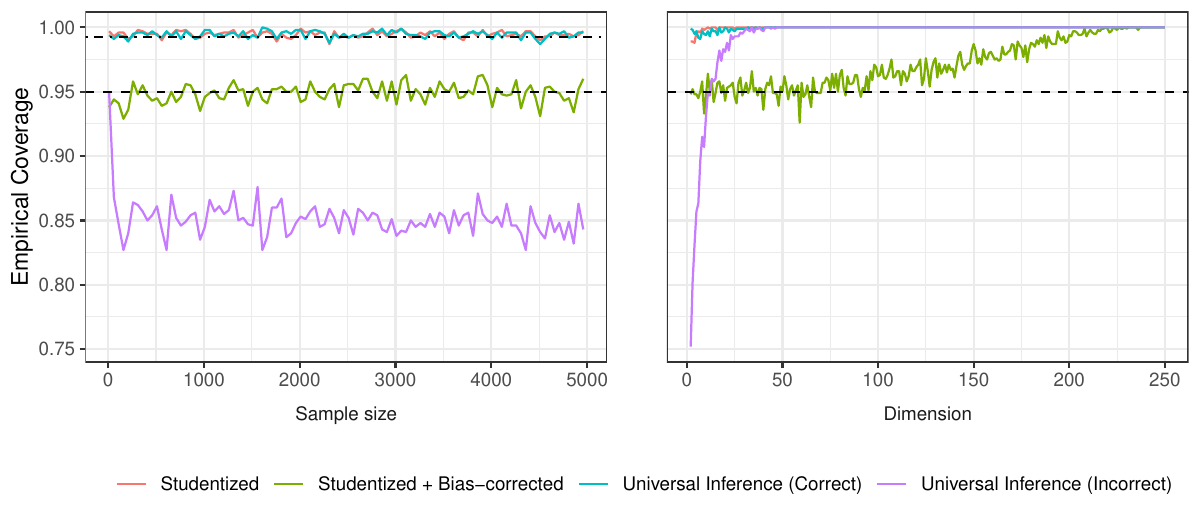}
  \label{fig:sub1}
\end{subfigure}
\caption{Comparison of the empirical coverage of $95\%$ confidence sets for fixed dimension with varying sample size (left panel) and fixed sample size with varying dimension (right panel). The empirical coverage is computed from 1000 replications. The methods are labeled as: \emph{Universal Inference (Correct)} based on \eqref{eq:ui-def-model}, \emph{Universal Inference (Incorrect)} based on \eqref{eq:ui-def-incorrect}, \emph{Studentized} based on 
 \eqref{eq:std-ui-lr}, and \emph{Studentized + Bias-corrected} based on \eqref{eq:bc-ui-lr}. The left panel shows that as the sample size increases, both \emph{Universal Inference (Correct)} and \emph{Studentized} are conservative, \emph{Universal Inference (Incorrect)} is invalid, and \emph{Studentized + Bias-corrected} achieves nominal coverage. The right panel shows that as the dimension increases, \emph{Universal Inference (Correct)} and \emph{Studentized} are extremely conservative with the coverage probability of $1$, \emph{Universal Inference (Incorrect)} transitions from invalid to extremely conservative, and \emph{Studentized + Bias-corrected} transitions from nominal to extremely conservative. These findings align with the theoretical expectations discussed in the manuscript.}
\label{fig:coverage}
\end{figure}

\section{Summary and Conclusions}
This manuscript presents new theoretical results on universal inference under model misspecification, focusing on the precise gap between its validity and conservatives. The corresponding analyses identify the regime where the method becomes extremely conservative. To address this issue, we propose a general remedy based on studentization and bias correction, which improves coverage accuracy when the product of the estimation error of two unknowns is negligible, exhibiting an intriguing double robustness property. Since such rate conditions often depend on the growth of the dimensions of the parameter space, avoiding extreme conservativeness is likely a dimension-dependent challenge without additional assumptions.

We anticipate that the studentization and bias correction approach could be extended to a broader class of methods, particularly those based on e-values, where validity is typically established via Markov’s inequality or its derivatives. While our current analysis relies on the CLT and its quantitative refinements, namely the Berry-Esseen bound, within an IID setting, many e-value procedures are designed for non-IID applications. Extending these techniques to non-IID data streams presents a promising avenue for future research. However, implementing effective bias correction in these settings is expected to be significantly more challenging.
 
\paragraph{Acknowledgments}
The author gratefully acknowledges Arun Kumar Kuchibhotla for proposing the idea of analyzing universal inference through the lens of the Berry-Esseen bound. The author also sincerely appreciates his numerous suggestions and corrections, as well as his encouragement to maintain this paper as a single-author work. The author wishes to thank Sivaraman Balakrishnan and Larry Wasserman for commenting on the manuscript. The author also acknowledges Woonyoung Chang for helpful discussions related to Section~\ref{supp:self-center} in the supplement.

\bibliographystyle{apalike}
\bibliography{ref.bib}
\newpage

\setcounter{section}{0}
\setcounter{equation}{0}
\setcounter{figure}{0}
\renewcommand{\thesection}{S.\arabic{section}}
\renewcommand{\theequation}{E.\arabic{equation}}
\renewcommand{\thefigure}{A.\arabic{figure}}
\renewcommand{\theHsection}{S.\arabic{section}}
\renewcommand{\theHequation}{E.\arabic{equation}}
\renewcommand{\theHfigure}{A.\arabic{figure}}

 \begin{center}
  \Large {\bf Supplement to ``On the Precise Asymptotics of Universal Inference''}
  \end{center}
       \vspace{1mm}
\begin{abstract}
This supplement contains the proofs of all the main results in the paper and some supporting lemmas.
\end{abstract}
\vspace{1mm}
\section*{Notation}
The table below summarizes the common objects we frequently refer to.
\begin{table}[h]
\centering
\begin{tblr}{
  hline{1-2, 7} = {-}{},
}
\textbf{Notation} & \textbf{Description}                  \\
$G$      & The data-generating distribution. \\
$\mathcal{G}$         & The class of distributions in which $G$ is contained.                             \\
$\mathcal{P}_{\Theta}$ 
         &  The working parametric model, which \emph{may not} contain $G$. \\            
$P_{\theta_G}$ & The KL-projection of $G$ onto $\mathcal{P}_{\Theta}$.\\
$\theta_G$ & The inference of interest, corresponding to the parameter for $P_{\theta_G}$.
\end{tblr}
\caption{Summary of notation used throughout the main text of the manuscript.}\label{tab:notation}
\end{table}
\section{Proof of Theorem~\ref{thm:ui-conservatism}}\label{supp:proof-ui}
Throughout, let $n$ denote the cardinality of the second data $D_2$. Conditioning on $D_1$, universal inference provides the following confidence set:
\begin{align*}
    \mathrm{CI}^{\mathrm{UI}}_{n, \alpha} &=\left\{\theta \in \Theta\, :\, \sum_{i\in D_2} \log \frac{p_{\widehat\theta_1}(Z_i)}{p_{\theta}(Z_i)}\le \log(1/\alpha) \right\} \\
    &= \left\{\theta \in \Theta\, :\, \mathbb{G}^\times_n [\sigma^{-1}_{\theta, \widehat\theta_1}\xi]\le n^{1/2}\sigma^{-1}_{\theta, \widehat\theta_1}\left(n^{-1}\log(1/\alpha) - \E_G\left[\log \frac{p_{\widehat\theta_1}}{p_{\theta}}\bigg|D_1\right]\right)\right\}
\end{align*}
where $\xi = \log p_{\widehat\theta_1}(Z) - \log p_{\theta}(Z)$ and $\sigma^2_{\theta, \widehat\theta_1} = \mathrm{Var}_G[\xi|D_1]$. The conditional empirical process is defined as \eqref{eq:cond-Gn}. Let us define the Kolmogorov-Smirnov distance for each $G$ as
\begin{align*}
    \Delta_{n, G}^{\mathrm{BE}} := \sup_{t\in\mathbb{R}}\left|\mathbb{P}_G(\mathbb{G}^\times_n [\sigma^{-1}_{\theta_G, \widehat\theta_1}\xi
    ] \le t|D_1) - \Phi(t)\right|
\end{align*}
where $\Phi(\cdot)$ is the cumulative distribution function of the standard normal random variable. 

Now, we denote $\widehat\theta_1 = \theta_G + h$. For any data-generating distribution $G$, the nonasymptotic pointwise miscoverage probability of universal inference is given by
\begin{align*}
    &\mathbb{P}_{G}(\theta_G \not\in \mathrm{CI}^{\mathrm{UI}}_{n, \alpha}|D_1)\\
    &\qquad = \mathbb{P}_{G}\left(\mathbb{G}^\times_n [\sigma^{-1}_{\theta_G, \theta_G + h}\xi] >  n^{1/2}\left(n^{-1}\log(1/\alpha) - \E_G\left[\log \frac{p_{\theta_G + h}}{p_{\theta_G}}\bigg|D_1\right]\right)\bigg|D_1\right) \\
    &\qquad= 1-\mathbb{P}_{G}\left(\mathbb{G}^\times_n [\sigma^{-1}_{\theta_G, \theta_G + h}\xi] \le  n^{1/2}\sigma^{-1}_{\theta_G, \theta_G + h}\left(n^{-1}\log(1/\alpha) - \E_G\left[\log \frac{p_{\theta_G + h}}{p_{\theta_G}}\bigg|D_1\right]\right) \bigg|D_1\right) \\
    & \qquad\qquad + \Phi\left(n^{-1/2}\sigma^{-1}_{\theta_G, \theta_G + h}\log(1/\alpha) - n^{1/2}\sigma^{-1}_{\theta_G, \theta_G + h}\E_G\left[\log \frac{p_{\theta_G + h}}{p_{\theta_G}}\bigg|D_1\right]\right)\\
    & \qquad\qquad - \Phi\left(n^{-1/2}\sigma^{-1}_{\theta_G, \theta_G + h}\log(1/\alpha) - n^{1/2}\sigma^{-1}_{\theta_G, \theta_G + h}\E_G\left[\log \frac{p_{\theta_G + h}}{p_{\theta_G}}\bigg|D_1\right]\right).
\end{align*}
Hence, we obtain 
\begin{align}\label{eq:thm1-mid-step}
    \left|\, \mathbb{P}_{G}(\theta_G \not\in \mathrm{CI}^{\mathrm{UI}}_{n, \alpha}|D_1) - \left(1-\Phi\left(\frac{\log(1/\alpha)}{n^{1/2}\sigma_{\theta_G, \theta_G + h}} - \frac{n^{1/2}}{\sigma_{\theta_G, \theta_G + h}}\E_G\left[\log \frac{p_{\theta_G + h}}{p_{\theta_G}}\bigg|D_1\right]\right)\right)\, \right| \le \Delta_{n, G}^{\mathrm{BE}}.
\end{align}
Let $\{\delta_n\}_{n=1}^\infty$ be a sequence used to define \ref{as:QMD}, \ref{as:hessian} and \ref{as:holder_inequality}. For any $\delta \in (0, \delta_n]$, we define an indicator function 
\begin{align}\label{eq:indicator1}
    1_\delta(h) := \begin{cases}
        1, & \text{if} \quad \max\left\{\|h\|_{I_{G}}, \|h\|_{V_{G}}, \|h\|^2_{V_{G}}/\|h\|_{I_{G}}\right\}\le \delta\\
        0 & \text{else}.
    \end{cases}
\end{align}
We split the analysis, depending on the evaluation of this indicator function. First, we focus on the case when this evaluates to one. In other words, we assume that 
\[\max\left\{\|h\|_{I_{G}}, \|h\|_{V_{G}}, \|h\|^2_{V_{G}}/\|h\|_{I_{G}}\right\}\le \delta.\]
In the following, we use a series of elementary results to relate two Gaussian CDFs that are ``close''. Specifically, let $q \in [1-\varepsilon, 1+\varepsilon]$ for some $\varepsilon$ such that $1-\varepsilon > 0$. Then, by Lemma 3(i) of \citet{korolev2017bounds}, it follows
\begin{align*}
    \sup_{t}\,|\Phi(q t) - \Phi(t)| \le \min\left\{1, \frac{2\varepsilon}{\sqrt{2\pi e}}\right\}.
\end{align*}
Furthermore, Lemma~\ref{lemma:gaussian-approx2} below also provides the bound 
\begin{align*}
    \sup_{t_1, t_2}|\Phi(t_1 + qt_2)-\Phi(t_1 + t_2)| \le \min\left\{1, \frac{2\varepsilon}{\sqrt{2\pi e}}\right\}.
\end{align*}
Under assumptions \ref{as:QMD}, \ref{as:hessian} and \ref{as:holder_inequality}, Lemma~\ref{lemma:variance} implies 
\begin{align*}
\left|\|h\|^{-2}_{I_{G}}\sigma^2_{\theta_G, \theta_G + h} -1\right| \le r_n(h)
\end{align*}
where the exact expression for $h \mapsto r_n(h)$ is provided in the statement of Lemma~\ref{lemma:variance}. Noting that for $\|h\|^{-1}_{I_{G}}\sigma_{\theta_G, \theta_G + h} \ge 0$,
\begin{align*}
\left|\|h\|^{-1}_{I_{G}}\sigma_{\theta_G, \theta_G + h} -1\right| \le \left|\|h\|^{-1}_{I_{G}}\sigma_{\theta_G, \theta_G + h} -1\right|\left|\|h\|^{-1}_{I_{G}}\sigma_{\theta_G, \theta_G + h} +1\right| \le r_n(h),
\end{align*}
and hence assuming $r_n(h) < 1$ (without loss of generality), this implies that $\|h\|^{-1}_{I_{G}}\sigma_{\theta_G, \theta_G + h} \in [1-r_n(h), 1+r_n(h)]$. It then follows by Lemma~\ref{lemma:gaussian-approx},
\begin{align*}
    &\left|\Phi\left(\frac{\log(1/\alpha)}{n^{1/2}\sigma_{\theta_G, \theta_G + h}} - \frac{n^{1/2}}{\sigma_{\theta_G, \theta_G + h}}\E_G\left[\log \frac{p_{\theta_G + h}}{p_{\theta_G}}\bigg|D_1\right]\right)\right. \\
    &\left.\qquad-\Phi\left(\frac{\log(1/\alpha)}{n^{1/2}\|h\|_{I_{G}}} - \frac{n^{1/2}}{\|h\|_{I_{G}}}\E_G\left[\log \frac{p_{\theta_G + h}}{p_{\theta_G}}\bigg|D_1\right]\right)\right| \le \min\left\{1, \frac{2r_n(h)}{\sqrt{2\pi e}}\right\}.
\end{align*}
Next, \ref{as:hessian} implies that 
\begin{align*}
    \left|-\left(\frac{1}{2}\|h\|^2_{V_{G}}\right)^{-1}\E_G\left[\log \frac{p_{\theta_G + h}}{p_{\theta_G}}\bigg|D_1\right]-1\right| \le \frac{2\omega_{G}(\|h\|_{V_{G}})}{\|h\|^2_{V_{G}}}.
\end{align*}
Assuming that $2\omega_{G}(\|h\|_{V_{G}})/\|h\|^2_{V_{G}} < 1$ (again, without loss of generality), Lemma~\ref{lemma:gaussian-approx2} implies the following bound:
\begin{align*}
    &\left|\Phi\left(\frac{\log(1/\alpha)}{n^{1/2}\|h\|_{I_{G}}} - \frac{n^{1/2}}{\|h\|_{I_{G}}}\E_G\left[\log \frac{p_{\theta_G + h}}{p_{\theta_G}}\bigg|D_1\right]\right)-\Phi\left(\frac{\log(1/\alpha)}{n^{1/2}\|h\|_{I_{G}}} + \frac{n^{1/2}\|h\|^2_{V_{G}}}{2\|h\|_{I_{G}}}\right)\right|\\
    &\qquad \le \min\left\{1, \frac{4\omega_{G}(\|h\|_{V_{G}})}{\sqrt{2\pi e}\|h\|^2_{V_{G}}}\right\}.
\end{align*}
Returning to \eqref{eq:thm1-mid-step}, the application of a series of triangle inequalities implies
\begin{align*}
    &\left|\, \mathbb{P}_{G}(\theta_G \not\in \mathrm{CI}^{\mathrm{UI}}_{n, \alpha}|D_1) - \left(1-\Phi\left(\frac{\log(1/\alpha)}{n^{1/2}\|h\|_{I_{G}}} + \frac{n^{1/2}\|h\|^2_{V_{G}}}{2\|h\|_{I_{G}}}\right)\right)\, \right| \\
    &\qquad \le \min\left\{1, \Delta_{n, G}^{\mathrm{BE}}+\left(\frac{2r_n(\delta)}{\sqrt{2\pi e}}+\frac{4\omega_{G}(\delta)}{\sqrt{2\pi e}\delta^2}\right)\right\} \quad \text{when} \quad 1_\delta(h) = 1.
\end{align*}
When $1_\delta(h)$ evaluates zero, we use the trivial upper bound of one. Hence, combining with the case with $1_\delta(h) = 0$, we obtain 
\begin{align*}
    &\left|\, \mathbb{P}_{G}(\theta_G \not\in \mathrm{CI}^{\mathrm{UI}}_{n, \alpha}|D_1) - \left(1-\Phi\left(\frac{\log(1/\alpha)}{n^{1/2}\|h\|_{I_{G}}} + \frac{n^{1/2}\|h\|^2_{V_{G}}}{2\|h\|_{I_{G}}}\right)\right)\, \right| \\
    &\qquad \le \min\left\{1, \Delta_{n, G}^{\mathrm{BE}}+\left(\frac{2r_n(\delta)}{\sqrt{2\pi e}}+\frac{4\omega_{G}(\delta)}{\sqrt{2\pi e}\delta^2}\right) + (1-1_\delta(h))\right\}.
\end{align*}
In particular, we can assume $r_n(h) \le 1$, without loss of generality, which simplifies the expression to 
\begin{align*}
    r_n(\delta) \le C_{p}\left\{\left(\frac{\phi_G(\delta)}{\delta^2}\right)^{1/2} + \delta^2 + \left(\max\left\{L_{n,G}, 1\right\}\nu_{G,p}(\delta)\right)^{1/2}\right\}
\end{align*}
where $C_{p}$ is a constant only depending on $p$. We conclude the first claim by taking the expectation over $D_1$. In particular,
\begin{align*}
    \E_G[1-1_\delta(h)] = \mathbb{P}_G\left(\max\left\{\|\widehat\theta_1- \theta_G\|_{I_{G}}, \|\widehat\theta_1- \theta_G\|_{V_{G}}, \frac{\|\widehat\theta_1 - \theta_G\|^2_{V_{G}}}{\|\widehat\theta_1 - \theta_G\|_{I_{G}}}\right\}> \delta\right) = \tau^{\mathrm{C}}_{n,G}(\delta)
\end{align*}
and we can take the infimum over $\delta \in (0, \delta_n]$ as the choice of $\delta$ was arbitrary.

The second claim follows by establishing the convergence of the remainder terms. The convergence of the remainder terms, except for $\Delta_{n, G}^{\mathrm{BE}}$, follows from assumptions \ref{as:QMD}---\ref{as:holder_inequality}. It remains to show that 
\[\limsup_{n\to\infty}\,\sup_{G\in \mathcal{G}}\,\E_G\left[\Delta_{n, G}^{\mathrm{BE}}\right] = 0\]
This claim follows directly from Lemma~\ref{lemma:KS-dist} under \ref{as:QMD}---\ref{as:ULC}, which concludes the result.

\section{Proofs of Corollaries~\ref{cor:trivial-conservative} and \ref{cor:trivial-correct}}\label{supp:proof-cor}
\begin{proof}[\bfseries{Proof of Corollary~\ref{cor:trivial-conservative}}]
    From Theorem~\ref{thm:ui-conservatism}, we have
    \begin{align*}
        \mathbb{P}_{G}(\theta_G \not\in \mathrm{CI}^{\mathrm{UI}}_{n, \alpha}) \le  1-\E_G\left[\Phi\left(\frac{\log(1/\alpha)}{n^{1/2}\|h\|_{I_{G}}} + \frac{n^{1/2}\|h\|^2_{V_{G}}}{2\|h\|_{I_{G}}}\right)\right] + \E_G[\textrm{Rem}_{n,G}]
    \end{align*}
    for some remainder term $\textrm{Rem}_{n,G}$ where $\sup_{G\in\mathcal{G}}\,\E_G[\textrm{Rem}_{n,G}] \to 0$ under \ref{as:QMD}---\ref{as:ULC}. We omit the precise definition of the remainder term as it is long. To claim that universal inference is extremely conservative as in Definition~\ref{def:uniformly-conservative}, we need to establish 
    \begin{align}\label{eq:cor1-claim}
        \inf_{G\in\mathcal{G}}\,\E_G\left[\Phi\left(\frac{\log(1/\alpha)}{n^{1/2}\|\widehat \theta_1 - \theta_G\|_{I_{G}}} + \frac{n^{1/2}\|\widehat \theta_1 - \theta_G\|^2_{V_{G}}}{2\|\widehat \theta_1 - \theta_G\|_{I_{G}}}\right)\right] \to 1.
    \end{align}
    For any $\varepsilon > 0$, let $\Omega_\varepsilon$ be an event defined as 
    \begin{align*}
     \Omega_\varepsilon = \left\{\min\left\{n^{1/2}\|\widehat \theta_1 - \theta_G\|_{I_{G}}, \frac{\|\widehat \theta_1 - \theta_G\|_{I_{G}}}{n^{1/2}\|\widehat \theta_1 - \theta_G\|^2_{V_{G}}}\right\} < \varepsilon \right\}.
    \end{align*}
    We then observe
    \begin{align*}
        \Omega_\varepsilon &= \left\{\frac{1}{\varepsilon} < \frac{1}{n^{1/2}\|\widehat \theta_1 - \theta_G\|_{I_{G}}} \right\} \cup \left\{\frac{1}{\varepsilon} < \frac{n^{1/2}\|\widehat \theta_1 - \theta_G\|^2_{V_{G}}}{\|\widehat \theta_1 - \theta_G\|_{I_{G}}}\right\}.
    \end{align*}
    Hence we establish for any $\varepsilon > 0$
    \begin{align*}
        &\E_G\left[\Phi\left(\frac{\log(1/\alpha)}{n^{1/2}\|\widehat \theta_1 - \theta_G\|_{I_{G}}} + \frac{n^{1/2}\|\widehat \theta_1 - \theta_G\|^2_{V_{G}}}{2\|\widehat \theta_1 - \theta_G\|_{I_{G}}}\right)\right] \\
        &\qquad \ge \min\left\{\Phi\left(\frac{\log(1/\alpha)}{\varepsilon}\right), \Phi\left( \frac{1}{2\varepsilon}\right)\right\}- \mathbb{P}_G(\Omega_\varepsilon^c).
    \end{align*}
    By assumption \eqref{eq:either-or}, $\sup_{G\in\mathcal{G}}\mathbb{P}_G(\Omega_\varepsilon^c) \to 0$ for any $\varepsilon > 0$. Finally, the claim is concluded by taking $\varepsilon \to 0$.
\end{proof}
\begin{proof}[\bfseries{Proof of Corollary~\ref{cor:trivial-correct}}]
From Theorem~\ref{thm:ui-conservatism}, we have
    \begin{align*}
        &\mathbb{P}_{G}(\theta_G \in \mathrm{CI}^{\mathrm{UI}}_{n, \alpha}) \ge  \E_G\left[\Phi\left(\frac{\log(1/\alpha)}{n^{1/2}\|h\|_{I_{G}}} + \frac{n^{1/2}\|h\|^2_{V_{G}}}{2\|h\|_{I_{G}}}\right)\right] - \E_G[\textrm{Rem}_{n,G}]\\
        &\Rightarrow \inf_{G\in \mathcal{G}}\mathbb{P}_{G}(\theta_G \in \mathrm{CI}^{\mathrm{UI}}_{n, \alpha}) \ge  \inf_{G\in \mathcal{G}}\E_G\left[\Phi\left(\frac{\log(1/\alpha)}{n^{1/2}\|h\|_{I_{G}}} + \frac{n^{1/2}\|h\|^2_{V_{G}}}{2\|h\|_{I_{G}}}\right)\right] - \sup_{G\in \mathcal{G}}\E_G[\textrm{Rem}_{n,G}]
    \end{align*}
    for some remainder term $\textrm{Rem}_{n,G}$ where $\sup_{G\in\mathcal{G}}\,\E_G[\textrm{Rem}_{n,G}] \to 0$ under \ref{as:QMD}---\ref{as:ULC}. We omit the precise definition of the remainder term.
    
    Let $M_G =n^{1/2}\|\widehat \theta_1 - \theta_G\|_{I_{G}}$ before taking any limit. Then, the remainder term $r^{\mathrm{UI}}_{n,G}$ in Theorem~\ref{thm:ui-conservatism} can be uniformly bounded from below for all $G \in \mathcal{G}$ such that 
    \begin{align*}
        &\frac{\log(1/\alpha)}{n^{1/2}\|\widehat \theta_1 - \theta_G\|_{ I_{G}}} +n^{1/2}\|\widehat \theta_1 - \theta_G\|_{ I_{G}}\frac{\|\widehat \theta_1 - \theta_G\|_{ V_{G}}^2}{2\|\widehat \theta_1 - \theta_G\|_{ I_{G}}^2} \\
        &\qquad = \frac{\log(1/\alpha)}{M_G } +M_G\frac{\|\widehat \theta_1 - \theta_G\|_{ V_{G}}^2}{2\|\widehat \theta_1 - \theta_G\|_{ I_{G}}^2}\ge \left(2\log(1/\alpha)\frac{\|\widehat \theta_1 - \theta_G\|_{ V_{G}}^2}{\|\widehat \theta_1 - \theta_G\|_{ I_{G}}^2}\right)^{1/2}.
    \end{align*}   
    This concludes the first result. 

    In particular, universal inference is provably conservative when
    \[\left(2\log(1/\alpha)\frac{\|\widehat \theta_1 - \theta_G\|_{ V_{G}}^2}{\|\widehat \theta_1 - \theta_G\|_{ I_{G}}^2}\right)^{1/2} \ge \sqrt{2\log(1/\alpha) \lambda_{\min}(I^{-1/2}_{G}V_{G} I^{-1/2}_{G})}> z_\alpha.\]
    When the model is correctly specified, that is $G \in \mathcal{P}_{\Theta}$, we have $I_{G} = V_{G}$, which implies $\lambda_{\min}(I^{-1/2}_{G}V_{G} I^{-1/2}_{G}) = 1$. This leads to the bound
    \begin{align*}
       \inf_{G \in \mathcal{G}}\,\Phi(z_\alpha + r^{\mathrm{UI}}_{n,G})  \ge \Phi(\sqrt{2\log(1/\alpha)})  > 1-\alpha,
    \end{align*}
    where the last inequality holds from Chernoff bound. Let $W \sim N(0,1)$. Then 
    \begin{align*}
        \mathbb{P}(W \ge t) \le \frac{\E[\exp(\lambda W)]}{\exp(\lambda t)} = \exp(-t^2/2).
    \end{align*}
    Choosing $t = \left(2\log(1/\alpha)\right)^{1/2}$, we obtain the inequality. Notably, this bound is strict unless $\alpha$ takes an extreme value. Thus, under the correct model specification, it follows that 
    \begin{align*}
        \E_G[\Phi(z_\alpha + r^{\mathrm{UI}}_{n,G})] -(1-\alpha) &\ge \inf_{G \in \mathcal{G}}\,\E_G[\Phi(z_\alpha + r^{\mathrm{UI}}_{n,G})] -(1-\alpha) \\&\ge \Phi(\sqrt{2\log(1/\alpha)} ) - (1-\alpha) > 0.
    \end{align*}
    This reveals that the universal confidence set is necessarily conservative when the model is correctly specified.
    \end{proof}
    \subsection{Asymptotic Results in Fixed Dimension}\label{supp:invalid}
    For simplicity, we focus on a pointwise result such that $\mathcal{G} = \{G\}$, but the result can be extended to a uniform statement. Now we assume that 
    \begin{align*}
        n^{1/2}I_{G}^{1/2}(\widehat \theta_1 - \theta_G) \overset{d}{\to} W_G
    \end{align*}
    for some limiting distribution $W_G$. Then by the dominated convergence theorem and the continuous mapping theorem, we have 
 \begin{align*}
        &\E_G\left[\Phi\left(\frac{\log(1/\alpha)}{\|W_G\|} +2^{-1}\|W_G\|\lambda_{\min}(I^{-1/2}_{G}V_{G} I^{-1/2}_{G})\right)\right] \le \lim_{n\to \infty}\E_G\left[\Phi\left(z_\alpha + r^{\mathrm{UI}}_{n,G}\right)\right] \\
        &\qquad \le \E_G\left[ \Phi\left(\frac{\log(1/\alpha)}{\|W_G\|} +2^{-1}\|W_G\|\lambda_{\max}(I^{-1/2}_{G}V_{G} I^{-1/2}_{G})\right)\right].
    \end{align*}
    Universal inference is conservative when 
    \begin{align}\label{eq:asym1}
        1-\alpha \le \E_G\left[\Phi\left(\frac{\log(1/\alpha)}{\|W_G\|} +2^{-1}\|W_G\|\lambda_{\min}(I^{-1/2}_{G}V_{G} I^{-1/2}_{G})\right)\right].
    \end{align}
    In particular, this inequality has no solution when 
    \[\sqrt{2\log(1/\alpha) \lambda_{\min}(I^{-1/2}_{G}V_{G} I^{-1/2}_{G})}> z_\alpha, \]
    and thus the inequality \eqref{eq:asym1} becomes strict.
    Similarly, universal inference can become invalid when 
    \begin{align*}
        \E_G\left[\Phi\left(\frac{\log(1/\alpha)}{\|W_G\|} +2^{-1}\|W_G\|\lambda_{\max}(I^{-1/2}_{G}V_{G} I^{-1/2}_{G})\right)\right] \le  1-\alpha. 
    \end{align*}
    which has a solution when 
    \begin{align*}
        \sqrt{2\log(1/\alpha) \lambda_{\max}(I^{-1/2}_{G}V_{G} I^{-1/2}_{G})}\le z_\alpha.
    \end{align*}
    On the event when 
    \begin{align*}
        &\|W_G\| \in \\
        & \left(\frac{z_\alpha - \sqrt{z_\alpha^2 -2\log(1/\alpha) \lambda_{\max}(I^{-1/2}_{G}V_{G} I^{-1/2}_{G})}}{\lambda_{\max}(I^{-1/2}_{G}V_{G} I^{-1/2}_{G})}, \frac{z_\alpha + \sqrt{z_\alpha^2 -2\log(1/\alpha) \lambda_{\max}(I^{-1/2}_{G}V_{G} I^{-1/2}_{G})}}{\lambda_{\max}(I^{-1/2}_{G}V_{G} I^{-1/2}_{G})}\right)
    \end{align*}
    universal inference has coverage strictly less than $1-\alpha$, and hence probably invalid. It is, however, difficult to establish the existence of an estimator that satisfies the above condition. 

    In general, it seems to require some miraculous alignment of a constant to claim the existence of an estimator that achieves
    \begin{align*}
        \lim_{n\to \infty}\E_G\left[\Phi\left(\frac{\log(1/\alpha)}{n^{1/2}\|\widehat \theta_1 - \theta_G\|_{I_{G}}} + \frac{n^{1/2}\|\widehat \theta_1 - \theta_G\|_{I_{G}}\|\widehat \theta_1 - \theta_G\|^2_{V_{G}}}{2\|\widehat \theta_1 - \theta_G\|^2_{I_{G}}}\right)\right] \to 1-\alpha.
    \end{align*}
    
\section{Proof of Theorem~\ref{thm:bc-conservatism}}\label{supp:proof-bc}
 Throughout, let $n$ denote the cardinality of the second data $D_2$. Similar to the proof of Theorem~\ref{thm:ui-conservatism}, we proceed by providing the bounds on the difference between conditional probability given $D_1$. Let $\{\delta_n\}_{n=1}^\infty$ be a sequence used to define \ref{as:QMD}, \ref{as:hessian} and \ref{as:holder_inequality}. For any $\delta \in (0, \delta_n]$, we define an indicator function $h \mapsto 1_\delta(h)$ as \eqref{eq:indicator1}. We first analyze the case when this function evaluates one.

Let $\{\gamma_n\}_{n=1}^\infty$ be a sequence used to define \ref{as:double-robust}. For any $\gamma \in (0, \gamma_n]$, define an event 
\begin{align*}
    \Omega_\gamma := \left\{ n^{1/2}\|\widehat \theta_1 - \theta_G\|_{I_{G}}|u^\top I^{-1/2}_{G}(\widehat V_{G}-V_{G})I^{-1/2}_{G}u| \le \gamma \right\} 
\end{align*}
Also, for any $\varepsilon \in (0,1]$, we define an event 
\begin{align*}
    \Omega_\varepsilon := \left\{ \left|\frac{\widehat\sigma^2_{\theta_G,\theta_G+h}}{\|h\|^2_{I_{G}}}-1\right|   \le \varepsilon^{-1}\sqrt{2}\left(\frac{\nu_{G,4}^2(\delta)}{n^{1/2}\delta^2} +r_n(\delta)\right)\right\}
\end{align*}
where $h \mapsto r_n(h)$ is defined in the statement of Lemma~\ref{lemma:variance} and $\nu_{G,4}$ is provided in \ref{as:fourth-moment}.

Conditioning on $D_1$, the studentized and bias-corrected universal inference provides the following confidence set:
\begin{align*}
    &\mathrm{CI}^{\mathrm{BC}}_{n, \alpha} = \left\{\theta \in \Theta \, : \, n^{-1}\sum_{i \in D_2}\, \log \frac{p_{\widehat\theta_1}(Z_i)}{p_{\theta}(Z_i)}+\frac{1}{2} (\widehat \theta_1 - \theta)^\top  \widehat V_G (\widehat \theta_1 - \theta) \le n^{-1/2}z_{\alpha} \widehat\sigma_{\theta, \widehat\theta_1}\right\}\\
    &= \left\{\theta \in \Theta\, :\, \mathbb{G}^\times_n [\sigma^{-1}_{\theta, \widehat\theta_1}\xi]\le \frac{\widehat\sigma_{\theta, \widehat\theta_1}}{\sigma_{\theta, \widehat\theta_1}}z_\alpha - \frac{n^{1/2}}{\sigma_{\theta, \widehat\theta_1}}\left(\frac{1}{2} (\widehat \theta_1 - \theta)^\top  \widehat V_G (\widehat \theta_1 - \theta)  + \E_G\left[\log \frac{p_{\widehat\theta_1}}{p_{\theta}}\bigg|D_1\right]\right)\right\}
\end{align*}
where $\xi = \log p_{\widehat\theta_1}(Z) - \log p_{\theta}(Z)$ and $\widehat\sigma^2_{\theta, \widehat\theta_1}$ is the sample variance of $\{\xi_i \, :\, i \in D_2\}$. Now, we denote $\widehat\theta_1 = \theta_G + h$. Then the conditional miscoverage rate can be bounded as 
\begin{align*}
    &\mathbb{P}_{G}(\theta_G \not\in \mathrm{CI}^{\mathrm{BC}}_{n, \alpha}|D_1) 
 \\
 &\qquad\le  \mathbb{P}_{G}\left(\mathbb{G}^\times_n [\sigma^{-1}_{\theta_G, \theta_G + h}\xi]> \frac{\widehat\sigma_{\theta_G, \theta_G + h}}{\sigma_{\theta_G, \widehat\theta_1}}z_\alpha - \frac{n^{1/2}}{\sigma_{\theta_G, \theta_G + h}}\left(\frac{1}{2} h^\top  (\widehat V_{G} -V_{G})h  \right)\right.\\
     &\qquad\qquad\qquad\qquad\qquad \left. - \frac{n^{1/2}}{\sigma_{\theta_G, \theta_G + h}}\left(\E_G\left[\log \frac{p_{\theta_G + h}}{p_{\theta_G}}\bigg|D_1\right]+\frac{1}{2} h^\top   V_{G} h\right) \cap \Omega_\gamma1_\delta(h) \cap \Omega_\varepsilon1_\delta(h)\right) \\
     &\qquad \qquad + \mathbb{P}_G(\Omega_\gamma^c1_\delta(h)|D_1)+ \mathbb{P}_G(\Omega_\varepsilon^c1_\delta(h)|D_1) \\
 &\qquad\le  1-\Phi\left(\left|1-\varepsilon^{-1}\sqrt{2}\left(\frac{\nu_{G,4}^2(\delta)}{n^{1/2}\delta^2} +r_n(\delta)\right)\right|z_\alpha - \frac{\|h\|_{I_G}\gamma}{2\sigma_{\theta_G, \theta_G + h}}\right.\\
     &\qquad\qquad\qquad\qquad\qquad\qquad\qquad \left. - \frac{n^{1/2}}{\sigma_{\theta_G, \theta_G + h}}\left(\E_G\left[\log \frac{p_{\theta_G + h}}{p_{\theta_G}}\bigg|D_1\right]+\frac{1}{2} h^\top   V_{G} h\right) \right)\\
     &\qquad \qquad + \Delta^{\mathrm{BE}}_{n,G} + \mathbb{P}_G(\Omega_\gamma^c1_\delta(h)|D_1)+ \mathbb{P}_G(\Omega_\varepsilon^c1_\delta(h)|D_1).
\end{align*}
By the analogous applications of Lemma \ref{lemma:gaussian-approx2} as in the proof of Theorem~\ref{thm:ui-conservatism}, we obtain 
\begin{align*}
    &\left|\Phi\left(\left|1-\varepsilon^{-1}\sqrt{2}\left(\frac{\nu_{G,4}^2(\delta)}{n^{1/2}\delta^2} +r_n(\delta)\right)\right|z_\alpha - \frac{\|h\|_{I_G}\gamma}{2\sigma_{\theta_G, \theta_G + h}}\right.\right.\\
     &\left.\qquad\qquad\qquad \left. - \frac{n^{1/2}}{\sigma_{\theta_G, \theta_G + h}}\left(\E_G\left[\log \frac{p_{\theta_G + h}}{p_{\theta_G}}\bigg|D_1\right]+\frac{1}{2} h^\top   V_{G} h\right) \right) - \Phi(z_\alpha + \gamma)\right| \\
     &\qquad \le \min\left\{1,  \frac{2r_n(\delta)}{\sqrt{2\pi e}} + \frac{2}{\varepsilon\sqrt{\pi e}}\left(\frac{\nu_{G,4}^2(\delta)}{n^{1/2}\delta^2} +r_n(\delta)\right)+\frac{4\omega_{G}(\delta)}{\sqrt{2\pi e}\delta^2}\right\}
\end{align*}
assuming \ref{as:QMD}, \ref{as:hessian}, \ref{as:holder_inequality} and \ref{as:fourth-moment}.
Also, by $(2\pi)^{-1/2}$-Lipschitz continuity of the Gaussian CDF, we obtain 
\begin{align*}
    |\Phi(z_\alpha + \gamma) - \Phi(z_\alpha)| \le \frac{\gamma}{\sqrt{2\pi}}.
\end{align*}

Repeating the analogous derivation for the lower bound, we obtain  
\begin{align*}
    &|\mathbb{P}_{G}(\theta_G \not\in \mathrm{CI}^{\mathrm{BC}}_{n, \alpha}|D_1) - \Phi(z_\alpha)| \\
    &\qquad \le\min\bigg\{1,  \Delta^{\mathrm{BE}}_{n,G} + \mathbb{P}_G(\Omega_\gamma^c1_\delta(h)|D_1)+ \mathbb{P}_G(\Omega_\varepsilon^c1_\delta(h)|D_1)\\
    &\qquad\qquad+\frac{2r_n(\delta)}{\sqrt{2\pi e}} + \frac{2}{\varepsilon\sqrt{\pi e}}\left(\frac{\nu_{G,4}^2(\delta)}{n^{1/2}\delta^2} +r_n(\delta)\right)+\frac{4\omega_{G}(\delta)}{\sqrt{2\pi e}\delta^2}+ \frac{\gamma}{\sqrt{2\pi}}\bigg\} \quad \text{when} \quad 1_\delta(h) = 1.
\end{align*}
When $1_\delta(h)$ evaluates zero, we use the trivial upper bound of one, which leads 
\begin{align*}
    &|\mathbb{P}_{G}(\theta_G \not\in \mathrm{CI}^{\mathrm{BC}}_{n, \alpha}|D_1) - \Phi(z_\alpha)| \\
    &\qquad \le\min\bigg\{1,  \Delta^{\mathrm{BE}}_{n,G} + \mathbb{P}_G(\Omega_\gamma^c1_\delta(h)|D_1)+ \mathbb{P}_G(\Omega_\varepsilon^c1_\delta(h)|D_1)\\
    &\qquad\qquad+\frac{2r_n(\delta)}{\sqrt{2\pi e}} + \frac{2}{\varepsilon\sqrt{\pi e}}\left(\frac{\nu_{G,4}^2(\delta)}{n^{1/2}\delta^2} +r_n(\delta)\right)+\frac{4\omega_{G}(\delta)}{\sqrt{2\pi e}\delta^2}+ \frac{\gamma}{\sqrt{2\pi}} + (1-1_\delta(h))\bigg\}.
\end{align*}

In particular, we have
\begin{align*}
    \mathbb{P}_G(\Omega_\varepsilon^c1_\delta(h)|D_1) &= \mathbb{P}_G\left( \left\{\left|\frac{\widehat\sigma^2_{\theta_G,\theta_G+h}}{\|h\|^2_{I_{G}}}-1\right|   > \varepsilon^{-1}\sqrt{2}\left(\frac{\nu_{G,4}^2(\delta)}{n^{1/2}\delta^2} +r_n(\delta)\right)\right\}1_\delta(h)\right) \le \varepsilon
\end{align*}
where the last inequality follows from Markov's inequality and Lemma \ref{lemma:var-consist} assuming \ref{as:QMD}, \ref{as:hessian}, \ref{as:holder_inequality}, and \ref{as:fourth-moment} as well as $1_\delta(h)=1$. Also, observe that 
\begin{align*}
\E_G[\mathbb{P}_G(\Omega_\gamma^c1_\delta(h)|D_1)] \le  \mathbb{P}_G\left(\,n^{1/2}\|\widehat \theta_1 - \theta_G\|_{I_{G}}|u^\top I^{-1/2}_{G}(\widehat V_{G}-V_{G})I^{-1/2}_{G}u| > \gamma \right) = \tau_{n,G}^{\mathrm{DR}}(\gamma)
\end{align*}
and
\begin{align*}
    \E_G[1-1_\delta(h)] = \mathbb{P}_G\left(\max\left\{\|\widehat\theta_1- \theta_G\|_{I_{G}}, \|\widehat\theta_1- \theta_G\|_{V_{G}}, \frac{\|\widehat\theta_1 - \theta_G\|^2_{V_{G}}}{\|\widehat\theta_1 - \theta_G\|_{I_{G}}}\right\}> \delta\right) = \tau^{\mathrm{C}}_{n,G}(\delta)
\end{align*}
as before. The first claim is concluded by taking the expectation over $D_1$, simplifying the expression for the upper bound by leveraging the fact that the bound can be taken less than or equal to 1 without loss of generality, and by taking infimum over $\gamma \in (0, \gamma_n]$, $\varepsilon \in (0,1]$ and $\delta \in (0,\delta_n]$. In particular, the terms involving $\varepsilon$ will be simplified since we have that 
\begin{align*}
    \inf_{\varepsilon \in (0,1]}\left\{\varepsilon + \frac{\mathcal{R}^{\mathrm{BC}}_{n,G}(\delta)}{\varepsilon}\right\} = 2\left(\mathcal{R}^{\mathrm{BC}}_{n,G}(\delta)\right)^{1/2}.
\end{align*}

The second claim follows by establishing the convergence of the remainder terms. The convergence of the remainder terms, except for $\Delta_{n, G}^{\mathrm{BE}}$, follows from assumptions \ref{as:QMD}---\ref{as:holder_inequality}, \ref{as:fourth-moment} and \ref{as:double-robust}. It remains to show that 
\[\limsup_{n\to\infty}\,\sup_{G\in \mathcal{G}}\,\E_G\left[\Delta_{n, G}^{\mathrm{BE}}\right] = 0\]
This claim follows directly from Lemma~\ref{lemma:KS-dist} under \ref{as:QMD}---\ref{as:ULC}. We thus conclude the result by taking the limit as $n \to \infty$.

\section{Additional Results without Bias Correction}\label{supp:std}
This section provides an additional result corresponding to Theorems~\ref{thm:ui-conservatism} and \ref{thm:bc-conservatism}, but for the studentized universal inference without bias correction. 
\begin{theorem}\label{thm:std-conservatism}
Assume \ref{as:QMD}---\ref{as:holder_inequality} and \ref{as:fourth-moment}. Define the Kolmogorov-Smirnov distance for $G \in \mathcal{G}$ as
\begin{align*}
    \Delta_{n, G}^{\mathrm{BE}} := \sup_{t\in\mathbb{R}}\left|\mathbb{P}_G(\mathbb{G}^\times_n [\sigma^{-1}_{\theta_G, \widehat\theta_1}\xi
    ] \le t | D_1) - \Phi(t)\right|
\end{align*}
where $\xi = \log p_{\widehat\theta_1}(Z) - \log p_{\theta_G}(Z)$ and $\sigma^2_{\theta_G, \widehat\theta_1} = \mathrm{Var}_{P}[\xi | D_1]$. 
    For any $G$, the nonasymptotic miscoverage probability of the confidence set satisfies 
      \begin{align*}
        &|\,\mathbb{P}_{G}(\theta_G \not\in \mathrm{CI}^{\mathrm{Std}}_{n, \alpha}) - \big(1-\E_G[\Phi(z_\alpha + r^{\mathrm{Std}}_{n,G}) \big)]\, |  \\
        &\qquad \le \min\bigg\{1, \E_G[\Delta_{n, G}^{\mathrm{BE}}]  +\inf_{\delta \in (0, \delta_n]}\left\{\left(\mathcal{R}^{\mathrm{Std}}_{n,G}(\delta)\right)^{1/2}+ \tau_{n,G}^{\mathrm{C}}(\delta)\right\}\bigg\}
    \end{align*} 
    where
    \begin{align}
        r^{\mathrm{Std}}_{n,G} :=2^{-1}n^{1/2}\|\widehat \theta_1 - \theta_G\|_{I_{G}} u^\top I^{-1/2}_{G}V_{G} I^{-1/2}_{G} u,
    \end{align}
    and
    \begin{align*}
        &\mathcal{R}^{\mathrm{Std}}_{n,G}(\delta) := C_{p}\left\{\delta^2 + \left(\frac{\phi_G(\delta)}{\delta^2}\right)^{1/2} + \frac{\omega_{G}(\delta)}{\delta^2} + \left(\max\left\{L_{n,G}, 1\right\}\nu_{G,p}(\delta)\right)^{1/2} + \frac{\nu_{G,4}^2(\delta)}{n^{1/2}\delta^2}\right\}
    \end{align*}
    with a constant $C_p$ only depending on $p$ in \ref{as:holder_inequality}. Additionally, suppose that the initial estimator $\widehat\theta_1$ satisfies \ref{as:initial_estimator} and the class of distributions $\mathcal{G}$ satisfies \ref{as:ULC}. Then   
    \begin{align*}
        \limsup_{n\to \infty}\, \sup_{G\in\mathcal{G}}\, |\,\mathbb{P}_{G}(\theta_G \not\in \mathrm{CI}^{\mathrm{Std}}_{n, \alpha}) - \big(1-\E_G[\Phi(z_\alpha + r^{\mathrm{Std}}_{n,G})]\big)\, | = 0.
    \end{align*}
\end{theorem}

\begin{proof}[\bfseries{Proof of Theorem \ref{thm:std-conservatism}}]
The proof is analogous to that of Theorem~\ref{thm:bc-conservatism}, up to the definition of the remainder term $r^{\mathrm{Std}}_{n,G}$. 
\end{proof}

\section{Supporting Lemmas}
\begin{lemma}\label{lemma:variance}
    Let $\{\delta_n\}_{n=1}^\infty$ be the sequence appearing in the definition of \ref{as:QMD}, \ref{as:hessian} and \ref{as:holder_inequality}. For $h \in \Theta$, assume $\max\left\{\|h\|_{I_{G}}, \|h\|_{V_{G}}, \|h\|^2_{V_{G}}/\|h\|_{I_{G}}\right\}\le \delta$ for some $\delta \in (0, \delta_n)$. Under \ref{as:QMD}, \ref{as:hessian}, and \ref{as:holder_inequality}, the conditional variance of log-likelihood ratio under $G$ is given as follows: 
    \begin{align*}
        &\left|\|h\|^{-2}_{I_{G}}\mathrm{Var}_G\left[\log \frac{p_{\theta_G+h}}{p_{\theta_G}}\mid D_1 \right]-1\right| \le r_n(\delta)
    \end{align*}
    where 
    \begin{equation}\label{eq:rn-def}
        \begin{aligned}
            r_n(h) &=\left(\frac{4\phi_G(\delta)}{\delta^2}\right)^{1/2}+ \frac{4\phi_G(\delta)}{\delta^2} + \frac{\delta^2}{4}\left(1 + \frac{4\omega_{G}^2(\delta)}{\delta^4}+\frac{4\omega_{G}(\delta)}{\delta^2}\right)\\
    &\qquad + 2^{1/p}C_p^{1/p}\left(1 + \frac{2\phi_G^{1/2}(\delta)}{\delta^2}\right) \left( \max\left\{L_{n,G}^2, 1\right\}\nu_{G,p}^{2}(\delta)\right)^{1/2}\\
    &\qquad+ 2^{2/p}C_p^{2/p} \max\left\{L_{n,G}^2, 1\right\}\nu_{G,p}^{2}(\delta)
        \end{aligned}
    \end{equation}
    for some constant $C_{p}$ only depending on $p$ corresponding to the one in \ref{as:holder_inequality}.
\end{lemma}
\begin{proof}[\bfseries{Proof of Lemma~\ref{lemma:variance}}]
    Let $\xi := \log p_{\theta_G +h}(Z) - \log p_{\theta_G}(Z)$ where $Z$ follows $G$. First, we obtain
\begin{align*}
    \|h\|^{-2}_{I_{G}}\mathrm{Var}_G[\xi|D_1] &= \|h\|^{-2}_{I_{G}}\E_G[\xi^2|D_1] - \|h\|^{-2}_{I_{G}}\left(\E_G[\xi|D_1]\right)^2
\end{align*}
and \ref{as:hessian} as well as the assumption $\|h\|_{V_G} \le \delta$ imply that 
\begin{align*}
    &\left(\E_G[\xi|D_1]\right)^2 = \left(\E_G[\xi|D_1]+\frac{1}{2} h^\top  V_{G} h - \frac{1}{2} h^\top  V_{G} h\right)^2\\
    &\quad\quad\quad =\left\{\left(\E_G[\xi|D_1]+\frac{1}{2} h^\top  V_{G} h\right)^2 -h^\top  V_{G} h\left(\E_G[\xi|D_1]+\frac{1}{2} h^\top  V_{G} h\right)+ \frac{1}{4} \|h\|_{V_{G}}^4\right\}\\
    & \Longleftrightarrow \left|\frac{4\left(\E_G[\xi|D_1]\right)^2}{\|h\|_{V_{G}}^4}-1\right| \le \frac{4\omega_G^2(\delta)}{\delta^4}+\frac{4\omega_G(\delta)}{\delta^2}.
\end{align*}
We now analyze the non-central second moment. We let $W = (p_{\theta_G +h}/p_{\theta_G})^{1/2}-1$. Then by the second-order Taylor expansion with an integral remainder term, we obtain 
\begin{align}\label{eq:g-remainder}
    \log(x+1) - x = -\int_{0}^x \frac{(x-t)}{(t+1)^2}\, dt = g(x).
\end{align}
Using this upper bound, we obtain 
\begin{align*}
    \E_G[\xi^2|D_1] = \E_G\left[\left(\log \frac{p_{\theta_G +h}}{p_{\theta_G}}\right)^2\bigg|D_1\right] &= 4\E_G\left[\left\{\log \left(\sqrt{\frac{p_{\theta_G +h}}{p_{\theta_G}}}-1+1\right)\right\}^2\bigg|D_1\right]\\
    &= 4\E_G\left[\left\{\log \left(W+1\right)-W+W\right\}^2\bigg|D_1\right]\\
    &= 4\E_G\left[\left(W+g(W)\right)^2|D_1\right]\\
    &= 4\E_G\left[W^2 + 2W g(W) + g(W)^2|D_1\right].
\end{align*}
We now analyze the leading term $\E_G[W^2|D_1]$. It first follows that 
\begin{align*}
    \E_G[W^2|D_1] &= \E_G\left[ \frac{\big(p^{1/2}_{\theta_G +h}-p^{1/2}_{\theta_G}\big)^2}{p_{\theta_G}} \bigg| D_1\right] \\
    &= \E_G \left[\frac{\left(p^{1/2}_{\theta_G +h}-p^{1/2}_{\theta_G} - \frac{1}{2}h^\top \dot\ell_{G}p^{1/2}_{\theta_G}+\frac{1}{2}h^\top \dot\ell_{G}p^{1/2}_{\theta_G}\right)^2}{p_{\theta_G}} \bigg|D_1\right] \\
    &= \E_G\left[ \frac{\left(p^{1/2}_{\theta_G +h}-p^{1/2}_{\theta_G} - \frac{1}{2}h^\top \dot\ell_{G}p^{1/2}_{\theta_G}\right)^2}{p_{\theta_G}} \bigg|D_1\right]\\
    &\qquad + 2\E_G\left[\int \frac{\left(p^{1/2}_{\theta_G +h}-p^{1/2}_{\theta_G} - \frac{1}{2}h^\top \dot\ell_{G}p^{1/2}_{\theta_G}\right)\left(\frac{1}{2}h^\top \dot\ell_{G}p^{1/2}_{\theta_G}\right)}{p_{\theta_G}} \bigg|D_1\right] + \frac{1}{4} \|h\|^2_{I_{G}}.
\end{align*}
Since $\|h\|_{I_G} \le \delta$, the first term is bounded by $\phi_G(\delta)$ under \ref{as:QMD}. The second term can be bounded by the Cauchy-Schwartz inequality as well as \ref{as:QMD}, which implies
\begin{align*}
    &\left|
    \E_G\left[\frac{\left(p^{1/2}_{\theta_G +h}-p^{1/2}_{\theta_G} - \frac{1}{2}h^\top \dot\ell_{G}p^{1/2}_{\theta_G}\right)\left(\frac{1}{2}h^\top \dot\ell_{G}p^{1/2}_{\theta_G}\right)}{p_{\theta_G}} \bigg|D_1\right]\right|\\
    &\qquad \le \frac{1}{2}\left(\E_G\left[ \frac{\left(p^{1/2}_{\theta_G +h}-p^{1/2}_{\theta_G} - \frac{1}{2}h^\top \dot\ell_{G}p^{1/2}_{\theta_G}\right)^2}{p_{\theta_G}} \bigg|D_1\right] \right)^{1/2}\|h\|_{I_{G}}= \frac{1}{2}\phi_G^{1/2}(\|h\|_{I_{G}})\|h\|_{I_{G}}.
\end{align*}
Using the results thus far, we have shown that 
\begin{align*}
   \left|\frac{4\E_G[W^2|D_1]}{\|h\|^2_{I_{G}}} - 1\right| \le  \left(\frac{16\phi_G(\delta)}{\delta^2}\right)^{1/2}+ \frac{4\phi_G(\delta)}{\delta^2}.
\end{align*}
Combining the earlier results, we have established 
\begin{align*}
    \left|\frac{\mathrm{Var}_G[\xi|D_1]}{\|h\|^2_{I_{G}}} -1\right| &\le \left(\frac{16\phi_G(\delta)}{\delta^2}\right)^{1/2}+ \frac{4\phi_G(\delta)}{\delta^2} \\
    &\qquad+ \frac{\|h\|_{V_{G}}^4}{4\|h\|^2_{I_{G}}}\left(1 + \frac{4\omega_G^2(\delta)}{\delta^4}+\frac{4\omega_G(\delta)}{\delta^2}\right)\\
    &\qquad + |\|h\|^{-2}_{I_{G}}\E_G\left[2W g(W) + g(W)^2 | D_1\right]|.
\end{align*}
We also recall that $\|h\|_{V_{G}}^4/\|h\|^2_{I_{G}} \le \delta^2$ is assumed. It remains to analyze the last term involving $g(W)$. First, we observe
\begin{align*}
    &|\|h\|^{-2}_{I_{G}}\E_G\left[2W g(W) + g(W)^2 | D_1\right]| \\
    &\qquad\le \|h\|^{-2}_{I_{G}}\left(\E_G\left[4W^2| D_1\right] \E_G\left[ g^2(W) | D_1\right] \right)^{1/2}+ \|h\|^{-2}_{I_{G}}\E_G\left[g(W)^2 | D_1\right]
\end{align*}
by the Cauchy-Schwartz inequality. Next, we observe that 
\begin{align}\label{eq:g-squared-bound}
    g(x)^2 = \left|\int_{0}^x \frac{(x-t)}{(t+1)}\, dt\right|^2 \le \left|\int_{0}^x \frac{x}{(t+1)}\, dt\right|^2= x^2 \log^2(1+x),
\end{align}
and hence we obtain for $p \ge 2$, corresponding to the one in \ref{as:holder_inequality},
\begin{align*}
    \E_G\left[g(W)^2 | D_1\right] &= \E_G\left[W^2 \log^2(1+W) | D_1\right] \\
    &\le \big(\E_G\left[|W|^{p} |D_1\right]\big)^{2/p}\big(\E_G\left[|\log(1+W)|^{(2p)/(p-2)} | D_1\right] \big)^{(p-2)/p}\\
    &= \frac{1}{4}\big(\E_G\left[|W|^{p} |D_1\right]\big)^{2/p}\left(\E_G\left[\left|\log\frac{p_{\theta_G+h}}{p_{\theta_G}}\right|^{(2p)/(p-2)}| D_1\right] \right)^{(p-2)/p}
\end{align*}
by the H\"{o}lder's inequality. We further observe that for some constant $C_p$, only depending on $p$ such that,
\begin{equation}\label{eq:Wp-derivation}
    \begin{aligned}
        \E_G\left[|W|^{p} |D_1\right] &= \E_G\left[\left|\sqrt{\frac{p_{\theta_G+h}}{p_{\theta_G}}}-1 + \frac{1}{2}h^\top \dot\ell_{G}-\frac{1}{2}h^\top \dot\ell_{G}\right|^{p} |D_1\right]\\
    &\le C_p \left\{\E_G\left[\left|\sqrt{\frac{p_{\theta_G+h}}{p_{\theta_G}}}-1+\frac{1}{2}h^\top \dot\ell_{G}\right|^{p}|D_1\right] + \E_G\left[\left|\frac{1}{2}h^\top \dot\ell_{G}\right|^{p} |D_1\right]\right\}\\
    &\le C_p \left\{L_{n,G}^p\|h\|_{I_G}^{p} + 2^{-p}\|h\|_{I_G} ^{p}\right\}\le 2C_p \max\left\{L_{n,G}^p, 1\right\}\|h\|_{I_G}^{p}
    \end{aligned}
\end{equation}
where we used \ref{as:holder_inequality}. Hence we arrive 
\begin{align}\label{eq:g-squared}
    \E_G\left[g(W)^2 | D_1\right] \le 2^{2/p}C_p^{2/p} \max\left\{L_{n,G}^2, 1\right\}\|h\|_{I_G}^{2}\nu_{G,p}^{2}(\delta)
\end{align}
under the assumption that assumption $\|h\|_{I_G} \le \delta$.
Putting together, 
\begin{align*}
    &|\|h\|^{-2}_{I_{G}}\E_G\left[2W g(W) + g(W)^2 | D_1\right]| \\
    &\qquad \le 2^{1/p}C_p^{1/p}\left(1 + \frac{2\phi_G^{1/2}(\delta)}{\delta}\right) \left( \max\left\{L_{n,G}^2, 1\right\}\nu_{G,p}^{2}(\delta)\right)^{1/2}\\
    &\qquad\qquad+ 2^{2/p}C_p^{2/p} \max\left\{L_{n,G}^2, 1\right\}\nu_{G,p}^{2}(\delta)
\end{align*}
for some constant $C_{p}$ only depending on $p$. This concludes the claim.
\end{proof}
\begin{lemma}\label{lemma:KS-dist}
Assume \ref{as:QMD}---\ref{as:ULC} hold for all $G \in \mathcal{G}$. Then, it holds that 
\begin{align*}
    \limsup_{n\to \infty}\,\sup_{G \in \mathcal{G}}\,  \E_G[\Delta_{n, G}^{\mathrm{BE}}] = 0.
\end{align*}
\end{lemma}
\begin{proof}[\bfseries{Proof of Lemma~\ref{lemma:KS-dist}}]
Let $\xi = \log p_{\widehat\theta_1}(Z) - \log p_{\theta_G}(Z)$ where $Z$ follows $G$.
Recall that the Kolmogorov-Smirnov distance for $G$ is defined as
\begin{align*}
    \Delta_{n, G}^{\mathrm{BE}} := \sup_{t\in\mathbb{R}}\left|\mathbb{P}_G(\mathbb{G}^\times_n [\sigma^{-1}_{\theta_G, \widehat\theta_1}\xi
    ] \le t) - \Phi(t)\right|.
\end{align*}

Now, we denote $\widehat\theta_1 = \theta_G + h$. Let $\{\delta_n\}_{n=1}^\infty$ be the sequence appearing in the definition of \ref{as:QMD}, \ref{as:hessian} and \ref{as:holder_inequality}. For any $\delta \in (0, \delta_n)$, we define the indicator function $h \mapsto 1_\delta(h)$ as in \eqref{eq:indicator1}. A classical result regarding the Berry-Esseen bound \citep{katz1963note} states that 
\begin{align}\label{eq:katz}
    \Delta_{n, G}^{\mathrm{BE}} \le \min\left\{1, C\E_G\left[\frac{|\xi-\E_G[\xi|D_1]|^2}{\mathrm{Var}_G[\xi|D_1]}\min\left\{1, \frac{|\xi-\E_G[\xi|D_1]|}{\mathrm{Var}^{1/2}_G[\xi|D_1]}\right\}\bigg|D_1\right] 1_\delta(h) + (1-1_\delta(h))\right\}
\end{align}
for some universal constant $C > 0$. Here, when the indicator function evaluates to zero, we are using the trivial upper bound of one. 

The upper bound to the right-hand term of \eqref{eq:katz} is provided by Proposition 1 of \citet{takatsu2025bridging}.  The proof thereby concludes as an application of Proposition 1. It remains to verify the condition of Proposition 1, which begins with finding a function $\varphi$ such that 
    \begin{align*}
        \frac{\E_G\big[\big\{(\log p_{\theta_G+h}-\log p_{\theta_G}) - \E_G[\log (p_{\theta_G+h}/p_{\theta_G})|D_1] - \langle h, \dot \ell_{G}(Z)\rangle\big\}^2|D_1\big]}{\E_G[\langle h, \dot \ell_{G}(Z)\rangle^2|D_1]} \le \varphi(\|h\|).
    \end{align*}
   Similarly to the proof of Lemma~\ref{lemma:variance}, we let $W = (p_{\widehat\theta_1}/p_{\theta_G})^{1/2}-1$. Then by the Taylor expansion of $\log(x+1) = x + g(x)$ where $g$ is defined in \eqref{eq:g-remainder}, we obtain 
\begin{align*}
    &\log \frac{p_{\widehat\theta_1}(Z_i)}{p_{\theta_G}(Z_i)} - \E_G\left[\log \frac{p_{\widehat\theta_1}}{p_{\theta_G}}\bigg|D_1\right]\\
    &\qquad = 2\log \left(\sqrt{\frac{p_{\widehat\theta_1}}{p_{\theta_G}}}(Z_i)-1+1\right) - \E_G\left[\log \frac{p_{\widehat\theta_1}}{p_{\theta_G}}\bigg|D_1\right]  = 2W_i + 2g(W_i)-\E_G\left[\log \frac{p_{\widehat\theta_1}}{p_{\theta_G}}\bigg|D_1\right].
\end{align*}
By letting $\widehat\theta_1 = \theta_G + h$, we obtain
\begin{align*}
    &\E_G \left[\left|\log \frac{p_{\theta_G+h}(Z_i)}{p_{\theta_G}(Z_i)} - \E_G\left[\log \frac{p_{\theta_G+h}}{p_{\theta}}\bigg|D_1\right] - h^\top \dot\ell_{G}\right|^2\bigg|D_1\right] \\
    &\qquad = \E_G \left[\left|2\left(W_i-\frac{1}{2} h^\top \dot\ell_{G}\right) + 2g(W_i) - \E_G\left[\log \frac{p_{\theta_G+h}}{p_{\theta_G}}\bigg|D_1\right]\right|^2\bigg|D_1\right]\\
    &\qquad \le C\left\{\E_G \left[\left|W_i-\frac{1}{2} h^\top \dot\ell_{G}\right|^2\bigg|D_1\right] + \E_G \left[g^2(W_i)|D_1\right]\right.\\
    &\left.\qquad \qquad + \left| \E_G\left[\log \frac{p_{\theta_G + h}}{p_{\theta_G}}\big |D_1\right] +\frac{1}{2} h^\top V_{G} h\right|^2 + \|h\|_{V_{G}}^4\right\} \\
    &\qquad \le C_p\left\{ \phi_G(\|h\|_{I_{G}}) + \max\left\{L_{n,G}^2, 1\right\}\|h\|_{I_{G}}^{2}\nu_{G,p}^{2}(\|h\|_{I_{G}})+\omega^2_G(\|h\|_{V_{G}})+ \|h\|_{V_{G}}^4\right\}
\end{align*}
for some constant $C_p$, only depending on $p$, where we use \ref{as:QMD}, \ref{as:hessian}, and the derivation leading up to \eqref{eq:g-squared}, which uses \ref{as:holder_inequality}. We also use the fact that we are only analyzing the case when $1_\delta(h) =1$. Putting together, we obtain 
\begin{align*}
     &\frac{\E_G [|\log p_{\theta_G+h}(Z_i) -\log p_{\theta_G}(Z_i) - \E_G[\log (p_{\theta_G+h}/p_{\theta_G})] - h^\top \dot\ell_{G}|^2]}{\E_G [|h^\top \dot\ell_{G}|^2]} 1_\delta(h)\\
     &\qquad \le C_p\left(\frac{\phi_G(\delta)}{\delta^2} + \max\left\{L_{n,G}^2, 1\right\}\nu_{G,p}^{2}(\delta)+\frac{\delta^2\omega^2_G(\delta)}{\delta^4} + \delta^2\right)1_\delta(h).
\end{align*}
Finally, by applying Proposition 1 of \citet{takatsu2025bridging} to the upper bound of \eqref{eq:katz}, we conclude that for any $\delta \in (0,\delta_n)$
\begin{align*}
    \sup_{G \in \mathcal{G}}\,\E_G[\Delta_{n, G}^{\mathrm{BE}}] & \lesssim \sup_{G \in \mathcal{G}}\,\frac{\phi_G(\delta)}{\delta^2} +\sup_{G \in \mathcal{G}}\, \max\left\{L_{n,G}^2, 1\right\}\nu_{G,p}^{2}(\delta)+ \delta^2\left(1 + \sup_{G \in \mathcal{G}}\,\frac{\omega^2_G(\delta)}{\delta^4}\right)  \\
    &\quad + \sup_{G \in \mathcal{G}}\,\mathbb{P}_G\left(\,\max\left\{\|\widehat\theta_1- \theta_G\|_{I_{G}}, \|\widehat\theta_1- \theta_G\|_{V_{G}}, \frac{\|\widehat\theta_1 - \theta_G\|^2_{V_{G}}}{\|\widehat\theta_1 - \theta_G\|_{I_{G}}}\right\}> \delta \right)\\
    &\quad +\sup_{G \in \mathcal{G}}\, \sup_{t \in \mathbb{R}^d}\,  \E_G\left[\frac{\langle t, \dot\ell_{G}\rangle^2}{\|\langle t,\dot\ell_{G} \rangle\|_{2}^2} \min\left\{1, \frac{|\langle t,\dot\ell_{G}\rangle|}{n^{1/2} \|\langle t, \dot\ell_{G}\rangle\|_{2}}\right\}\right].
\end{align*}
Finally, the last display converges zero as $n \to \infty$ assuming \ref{as:QMD}---\ref{as:ULC}. This concludes the claim.
\end{proof}
\begin{lemma}\label{lemma:var-consist}
Let $\{\delta_n\}_{n=1}^\infty$ be the sequence appearing in the definition of \ref{as:QMD}, \ref{as:hessian}, \ref{as:holder_inequality} and \ref{as:fourth-moment}. For $h \in \Theta$, assume $\max\left\{\|h\|_{I_{G}}, \|h\|_{V_{G}}, \|h\|^2_{V_{G}}/\|h\|_{I_{G}}\right\}\le \delta$ for some $\delta \in (0, \delta_n)$. Assume \ref{as:QMD}, \ref{as:hessian}, \ref{as:holder_inequality} and \ref{as:fourth-moment}. For any $\varepsilon > 0$, there exists a constant $C_\varepsilon$ only depending on $\varepsilon$ such that 
\begin{align*}
    &\left|\frac{\widehat\sigma^2_{\theta_G,  \theta_G+h}}{\|h\|^2_{I_{G}}}-1\right| \le C_{\varepsilon}\left(\frac{\nu_{G,4}^2(\delta)}{n^{1/2}\delta^2} +r_n(\delta)\right)
\end{align*}
with probability greater than $1-\varepsilon$, conditioning on $D_1$, where $h \mapsto r_n(h)$ is defined in \eqref{eq:rn-def}. Furthermore, 
\begin{align*}
    &\E_G\left[\left|\frac{\widehat\sigma^2_{\theta_G,\theta_G+h}}{\|h\|^2_{I_{G}}}-1\right|\right]  \le \sqrt{2}\left(\frac{\nu_{G,4}^2(\delta)}{n^{1/2}\delta^2} +r_n(\delta)\right).
\end{align*}
\end{lemma}
\begin{proof}[\bfseries{Proof of Lemma~\ref{lemma:var-consist}}]
First, by Markov inequality,
    \begin{align*}
\mathbb{P}_G\left(|\widehat\sigma^2_{\theta_G, \theta_G + h}/\|h\|^2_{I_{G}}-1| > \varepsilon | D_1\right) &= \mathbb{P}_G\left(|\widehat\sigma^2_{\theta_G, \theta_G + h}-\|h\|^2_{I_{G}}| > \varepsilon\|h\|^2_{I_{G}} | D_1\right) \\
&\le \frac{\sqrt{\E_G[|\widehat\sigma^2_{\theta_G, \theta_G + h}-\|h\|^2_{I_{G}}|^2 | D_1]}}{\varepsilon\|h\|^2_{I_{G}}}.
    \end{align*}
Since the sample variance is an unbiased estimator, 
\begin{align*}
   &\E_G[|\widehat\sigma^2_{\theta_G, \theta_G + h}-\sigma^2_{\theta_G, \theta_G + h}+\sigma^2_{\theta_G, \theta_G + h}-\|h\|^2_{I_{G}}|^2 | D_1] \\
   &\qquad\le 2\E_G[|\widehat\sigma^2_{\theta_G, \theta_G + h}-\E_G[\widehat\sigma^2_{\theta_G, \theta_G + h}| D_1]|^2 | D_1] + 2|\sigma^2_{\theta_G, \theta_G + h}-\|h\|^2_{I_{G}}|^2 \\
   &\qquad\le  2\mathrm{Var}_G[\widehat\sigma^2_{\theta_G, \theta_G + h} | D_1] + 2r^2_n(h)\|h\|^4_{I_{G}}
\end{align*}
where the last line uses Lemma~\ref{lemma:variance}, which uses \ref{as:QMD}, \ref{as:hessian} and \ref{as:holder_inequality} as well as that the indicator function $h \mapsto 1_\delta(h)$, defines in \eqref{eq:indicator1}, evaluates one. Here, $h\mapsto r_n(h)$ is defined in the statement of Lemma~\ref{lemma:variance}.
The variance of the sample variance is bounded such that 
\begin{align*}
    \mathrm{Var}_G[\widehat\sigma^2_{\theta_G, \theta_G+h} | D_1] \le \E_G[|\widehat\sigma^2_{\theta_G, \theta_G+h}-\sigma^2_{\theta_G, \theta_G+h}|^2 | D_1] \le \frac{\E_G[|\xi -\E_G[\xi| D_1]|^4 | D_1]}{n} 
\end{align*}
where $\xi = \log p_{\theta_G+h}(Z) - \log p_{\theta_G}(Z)$. Then, the following observation implies  
\begin{align*}
    &\E_G\left[\left|\log \frac{p_{\theta_G+h}(Z_i)}{p_{\theta_G}(Z_i)} - \E_G\left[\log \frac{p_{\theta_G+h}}{p_{\theta_G}}\bigg|D_1\right]\right|^4 \bigg| D_1\right] \\
    &\qquad \le 2\E_G\left[\left|\log \frac{p_{\theta_G+h}(Z_i)}{p_{\theta_G}(Z_i)}\right|^4\bigg| D_1\right]\le 2\nu_{G,4}^4(\|h\|_{I_{G}})
\end{align*}
where we use \ref{as:fourth-moment} as well as the assumption that $\|h\|_{I_G} \le \delta$ for some $\delta \in (0, \delta_n)$. This implies that 
\begin{align*}
    \mathbb{P}_G\left(|\widehat\sigma^2_{\theta_G, \theta_G+h}/\sigma^2_{\theta_G, \theta_G+h}-1| > \varepsilon | D_1\right) &\le \sqrt{\frac{2\nu_{G,4}^4(\|h\|_{I_{G}})}{n\varepsilon^2\|h\|_{I_{G}}^4}+\frac{2r^2(h)}{\varepsilon^2}}.
\end{align*}
Hence taking $\varepsilon^2 = C_\varepsilon (n^{-1}\|h\|_{I_{G}}^{-4}\nu_{G,4}^4(\|h\|_{I_{G}}) + r_n(\widehat\theta_1 - \theta_G))$, we conclude the first claim. The result in expectation is straightforward.
\end{proof}
We use the following results on the CDF for a Gaussian random variable. 
\begin{lemma}\label{lemma:gaussian-approx}
    Let $q > 0$ and $\Phi(\cdot)$ be the cumulative distribution function of the standard normal random variable. Then,
    \begin{align*}
        \sup_{t}|\Phi(qt)-\Phi(t)| \le \frac{1}{\sqrt{2\pi e}}\left(\max\left\{q,q^{-1}\right\}-1\right).
    \end{align*}
\end{lemma}
\begin{proof}[\bfseries{Proof of Lemma~\ref{lemma:gaussian-approx}}]
    See Lemma 3 of \citet{korolev2017bounds}.
\end{proof}
\begin{lemma}\label{lemma:gaussian-approx2}
    Let $q \in [1-\varepsilon, 1+\varepsilon]$ such that $1-\varepsilon > 0$. and $\Phi(\cdot)$ be the cumulative distribution function of the standard normal random variable. Then,
    \begin{align*}
        \sup_{t_1, t_2}|\Phi(t_1 + qt_2)-\Phi(t_1 + t_2)| \le \min\left\{1, \frac{2\varepsilon}{\sqrt{2\pi e}}\right\}.
    \end{align*}
\end{lemma}
\begin{proof}[\bfseries{Proof of Lemma~\ref{lemma:gaussian-approx2}}]
    For two differentiable distribution functions $F$ and $G$, the Kolmogorov-Smirnov distance $\sup_{t}|F(t)-G(t)|$ is attained at a stationary point $F'(t)=G'(t)$. Let $t_2$ be a fixed constant. Then 
    \begin{align*}
        \frac{d}{ dt}\Phi(t + qt_2) &= \Phi'(t + q t_2) = \frac{1}{\sqrt{2\pi}}\exp\left(-\left(\frac{t+ qt_2}{2}\right)^2\right)\quad\text{and}\\
         \frac{d}{ dt}\Phi(t + t_2) &= \Phi'(t_1 + t_2) = \frac{1}{\sqrt{2\pi}}\exp\left(-\left(\frac{t+ t_2}{2}\right)^2\right),
    \end{align*}
    and thus the difference is maximized at $t = (t_2+qt_2)/2$. Hence,  
    \begin{align*}
        \sup_{t_2}\sup_{t_1}|\Phi(t_1 + qt_2)-\Phi(t_1 + t_2)| =\sup_{t_2}\left|\Phi\left(\frac{t_2(1+3q)}{2}\right)-\Phi\left(\frac{t_2(3+q)}{2}\right)\right|.
    \end{align*}
    Therefore, Lemma~\ref{lemma:gaussian-approx} implies 
    \begin{align*}
        \sup_{t_1, t_2}|\Phi(t_1 + qt_2)-\Phi(t_1 + t_2)| &= \sup_{t_2}\left|\Phi\left(t_2\right)-\Phi\left(t_2\frac{3+q}{1+3q}\right)\right|\\
        &\le \frac{1}{\sqrt{2\pi e}}\left(\max\left\{\frac{3+q}{1+3q},\frac{1+3q}{3+q}\right\}-1\right)\\
        &= \frac{1}{\sqrt{2\pi e}}\left(\max\left\{\frac{2(1-q)}{1+3q},\frac{2(q-1)}{3+q}\right\}\right) \le \frac{2\varepsilon}{\sqrt{2\pi e}}.
    \end{align*}

\end{proof}
\section{On the Effect of Self-centering}\label{supp:self-center}
The analyses in this section are intentionally informal, as the result is specific to misspecified linear regression. The section is purposed to provide an intuitive explanation for the empirical phenomena observed in the simulation section. For the problem defined in Section~\ref{sec:illusration}, we observe that
\begin{align*}
    &n^{-1}\sum_{i \in D_2}(Y_i - \theta_G^\top X_i)^2 - (Y_i - \widehat\theta_1^\top X_i)^2 \\
    &\qquad = -n^{-1}\sum_{i \in D_2}2(\theta_G-\widehat\theta_1)^\top \varepsilon_i X_i - (\theta_G-\widehat\theta_1)^\top \left(n^{-1}\sum_{i \in D_2}X_i X_i^\top\right) (\theta_G-\widehat\theta_1).
\end{align*}
Thus, using $D_2$ to estimate the bias term $V_G$ leads to the \emph{exact} cancellation of the second-order term. The miscoverage probability of the studentized and bias-corrected confidence set for this problem, as defined in Section~\ref{sec:illusration}, has the following equivalent form:
\begin{equation}
    \begin{aligned}
        \mathbb{P}_G(\theta_G \not\in \mathrm{CI}^{\mathrm{BC}}_{n, \alpha}) = \mathbb{P}_G\bigg(-n^{-1}\sum_{i \in D_2}2(\theta_G-\widehat\theta_1)^\top \varepsilon_i X_i > n^{-1/2}z_{\alpha} \widehat\sigma_{\theta_G, \widehat\theta_1}\bigg).
    \end{aligned}
\end{equation}
In this case, the variance is over-estimated, especially when the second-order term is not negligible, leading to a valid but conservative confidence set when $d$ is comparable to $n$. This explains why the bias-corrected confidence set remains valid in Section~\ref{sec:illusration} for large $d$ while Theorem~\ref{thm:bc-conservatism} only provides a result for $d = o(n)$.

Meanwhile, one can consider a ``three-way" data splitting scheme, where each component, $\widehat \theta_1$, $\widehat V_G$ and $\mathrm{CI}^{\mathrm{BC}}_{n, \alpha}$, is derived from independent subsets of the data (see the right panel of Figure~\ref{fig:sample-splitting-2}). Under this scheme, the miscoverage probability of the confidence set based on studentization and bias correction takes the form:
\begin{equation}
    \begin{aligned}\label{eq:bc-three-way}
        &\mathbb{P}_G(\theta_G \not\in \mathrm{CI}^{\mathrm{BC}}_{n, \alpha}) \\
        &\qquad= \mathbb{P}_G\bigg(-n^{-1}\sum_{i \in D_3}2(\theta_G-\widehat\theta_1)^\top \varepsilon_i X_i \\
&\qquad\qquad\qquad + n^{-1} (\widehat \theta_1 - \theta_G)^\top  \bigg(\sum_{i \in D_2} X_iX_i^\top- \sum_{i \in D_3} X_iX_i^\top\bigg)(\widehat \theta_1 - \theta_G) > n^{-1/2}z_{\alpha} \widehat\sigma_{\theta_G, \widehat\theta_1}\bigg).
    \end{aligned}
\end{equation}
Here, the behavior of the resulting confidence set is expected to be more sensitive to parameter dimensions. While the corresponding simulation results are omitted from the manuscript, we have empirically observed that the confidence set based on \eqref{eq:bc-three-way} becomes invalid, i.e., undercovers the nominal level, when $d$ becomes comparable to $n$. This suggests that estimating the bias term $V_G$ using the same data as the construction of the confidence set provides an implicit benefit of additional bias correction through self-centering. The precise formulation of this statement is deferred to the future work.  
\begin{figure}
\centering
\begin{subfigure}{0.9\textwidth}
  \centering
  \includegraphics[width=\linewidth]{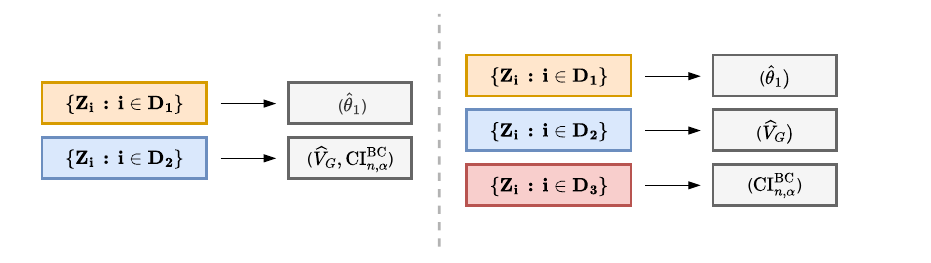}
\end{subfigure}
\caption{A schematic illustrating the sample splitting procedure. The arrow indicates that the objects are derived from the corresponding data. The left panel shows the proposed procedure, while the right panel provides the illustration used in this section}
\label{fig:sample-splitting-2}
\end{figure}

\section{Additional Illustration}\label{sup:laplace}
We provide additional results regarding the empirical coverage of four different confidence sets. For given $N \ge 2$ and $d$, we now generate the observations from the model:
\begin{align}\label{eq:exp-model2}
    Y_i = \theta_G^\top X_i + \epsilon_i  \quad\text{where}\quad X_i \sim N(0_d, I_d) \quad \text{and}\quad \epsilon_i \sim \mathrm{Lap}(0,1),
\end{align}
where $\mathrm{Lap}(0,1)$ denotes the Laplace distribution with the location parameter 0 and the scale parameter 1. The confidence sets are defined as in \eqref{eq:ui-def-model}, \eqref{eq:ui-def-incorrect}, \eqref{eq:std-ui-lr}, and \eqref{eq:bc-ui-lr}. In this case, both instances of universal inference are incorrectly specified since all methods mistakenly assume the likelihood follows a normal distribution.

Figure~\ref{fig:coverage2} displays the empirical coverage of the 95\% confidence sets, computed from 1000 replications. The four methods under consideration are labeled as follows: \emph{Universal Inference 1} corresponds to the confidence set in \eqref{eq:ui-def-model}, \emph{Universal Inference 2} to the confidence set in \eqref{eq:ui-def-incorrect}, and the remaining methods are labeled as before. As in Figure~\ref{fig:coverage}, the nominal level of $0.95$ is indicated by a dashed line. Additionally, we provide the theoretical coverage level $\Phi(\sqrt{2\log(1/\alpha)})$ with $\alpha=0.05$, shown as a dash-dotted line in the left panel.

The left panel of Figure~\ref{fig:coverage2} displays the empirical coverage for each method when $d=5$ and $10 \le N \le 5000$. We observe that \emph{Studentized} and \emph{Studentized + Bias-corrected} exhibit similar performance to the results in Figure~\ref{fig:coverage}, despite the clear model misspecification. On the other hand, \emph{Universal Inference 1} now achieves coverage closer to the $1-\alpha$ level. However, it is incorrect to interpret it as \emph{Universal Inference 1} performing better. This result should be seen as a cautionary tale, as it is possible to construct a data-generating distribution where universal inference-based methods perform nearly at the exact $1-\alpha$ level, even though no such implication should be drawn from the model. The right panel of Figure~\ref{fig:coverage2} displays the coverage for each method when $N=500$ and $2 \le d \le 250$ (with $n = |D_2| = 250$). The interpretation of these results is largely consistent with the analysis presented in the main text.

\begin{figure}
\centering
\begin{subfigure}{0.9\textwidth}
  \centering
  \includegraphics[width=\linewidth]{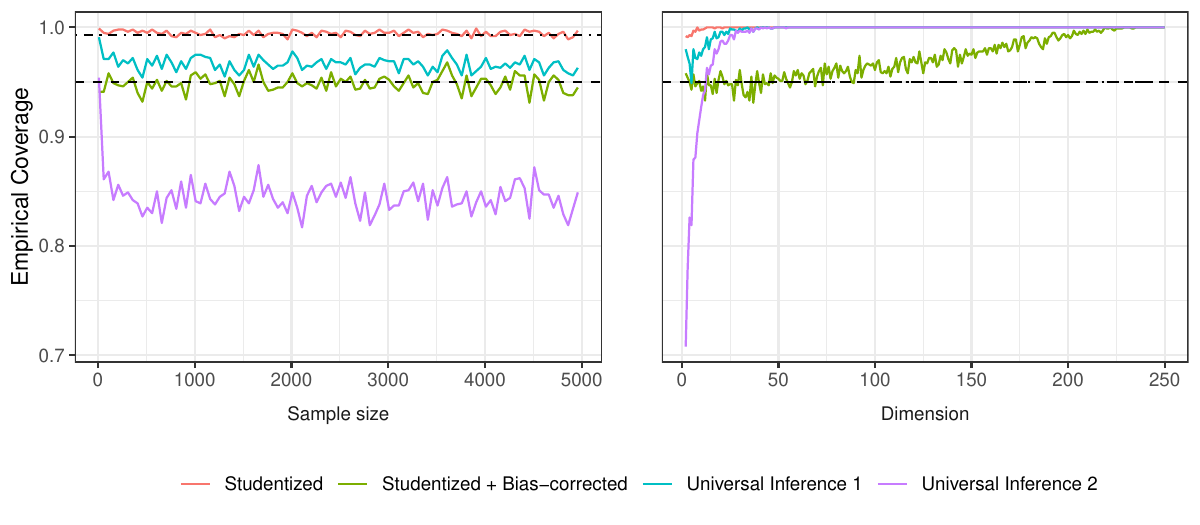}
  \label{fig:sub2}
\end{subfigure}
\caption{Comparison of the empirical coverage of $95\%$ confidence sets for fixed dimension with varying sample size (left panel) and fixed sample size with varying dimension (right panel). The empirical coverage is computed from 1000 replications. The methods are labeled as: \emph{Universal Inference 1} based on \eqref{eq:ui-def-model}, \emph{Universal Inference 2} based on \eqref{eq:ui-def-incorrect}, \emph{Studentized} based on 
 \eqref{eq:std-ui-lr}, and \emph{Studentized + Bias-corrected} based on \eqref{eq:bc-ui-lr}. It is incorrect to interpret that \emph{Universal Inference 1} performs better, as its coverage now approaches the nominal level. This merely reflects the fact that the validity of universal inference under model misspecification can be arbitrary, ranging from conservative, exact to invalid. While conservative, the validity of the \emph{Studentized} method remains robust under model misspecification, demonstrating the benefit of studentization even when bias correction is infeasible.}
\label{fig:coverage2}
\end{figure}
\end{document}